\pdfoutput=1
\PassOptionsToPackage{hidelinks,linktocpage,colorlinks=true,allcolors=blue,linktoc=all}{hyperref}
\PassOptionsToPackage{capitalise,nameinlink,noabbrev}{cleveref}

\PassOptionsToClass{11pt,letterpaper,onecolumn}{article}
\newif\ifusecolttemplate
\IfFileExists{includes/colt2025.cls}{%
  \IfFileExists{algorithm2e.sty}{\usecolttemplatetrue}{\usecolttemplatefalse}%
}{\usecolttemplatefalse}
\ifusecolttemplate
  \documentclass[11pt,final,cleveref]{includes/colt2025}
  \jmlrpages{}%
  \crefname{equation}{}{}
\else
  \documentclass[11pt,letterpaper,onecolumn]{article}
\fi

\usepackage[letterpaper,margin=1in]{geometry}
\AddToHook{env/abstract/begin}{\normalsize}

\usepackage{amsmath,amssymb,mathtools,mathrsfs}
\ifusecolttemplate\else
  \usepackage{amsthm}
\fi
\usepackage{xparse}
\usepackage{xargs}
\usepackage{xcolor}
\usepackage{tikz}
\usepackage{enumitem}
\usepackage{booktabs}
\usepackage{aliascnt}
\usepackage{url}
\ifusecolttemplate\else
  \usepackage[colorlinks=true,linkcolor=blue,citecolor=blue,urlcolor=blue]{hyperref}
  \usepackage[capitalise,nameinlink,noabbrev]{cleveref}
\fi

\usepackage{algorithm}
\usepackage{algcompatible}
\algnewcommand{\lst}{\texttt{lst}}
\algnewcommand{\slst}{\texttt{slst}}
\algnewcommand{\SEND}{\textbf{send}}

\newsavebox{\algleft}
\newsavebox{\algright}

\makeatletter
\newcounter{algorithmicH}
\let\oldalgorithmic\algorithmic
\renewcommand{\algorithmic}{%
  \stepcounter{algorithmicH}
  \oldalgorithmic}
\renewcommand{\theHALG@line}{ALG@line.\thealgorithmicH.\arabic{ALG@line}}
\makeatother


\usepackage[natbib, backend=biber, maxcitenames=3, minalphanames=3, maxbibnames=99, style=alphabetic, hyperref, backref, useprefix=true, uniquename=false, doi=false,url=false,eprint=false]{biblatex} 

\usepackage{csquotes}               
\bibliography{refs}        

\DeclareCiteCommand{\cite}
  {\usebibmacro{prenote}}
  {\usebibmacro{citeindex}%
   \printtext[bibhyperref]{\usebibmacro{cite}}}
  {\multicitedelim}
  {\usebibmacro{postnote}}

\DeclareCiteCommand*{\cite}
  {\usebibmacro{prenote}}
  {\usebibmacro{citeindex}%
   \printtext[bibhyperref]{\usebibmacro{citeyear}}}
  {\multicitedelim}
  {\usebibmacro{postnote}}

\DeclareCiteCommand{\parencite}[\mkbibparens]
  {\usebibmacro{prenote}}
  {\usebibmacro{citeindex}%
    \printtext[bibhyperref]{\usebibmacro{cite}}}
  {\multicitedelim}
  {\usebibmacro{postnote}}

\DeclareCiteCommand*{\parencite}[\mkbibparens]
  {\usebibmacro{prenote}}
  {\usebibmacro{citeindex}%
    \printtext[bibhyperref]{\usebibmacro{citeyear}}}
  {\multicitedelim}
  {\usebibmacro{postnote}}

\DeclareCiteCommand{\citeauthor}
  {\usebibmacro{prenote}}
  {\ifciteindex
     {\indexnames{labelname}}
     {}%
   \printtext[bibhyperref]{\printnames{labelname}}}
  {\multicitedelim}
  {\usebibmacro{postnote}}

\DeclareCiteCommand{\footcite}[\mkbibfootnote]
  {\usebibmacro{prenote}}
  {\usebibmacro{citeindex}%
  \printtext[bibhyperref]{ \usebibmacro{cite}}}
  {\multicitedelim}
  {\usebibmacro{postnote}}

\DeclareCiteCommand{\footcitetext}[\mkbibfootnotetext]
  {\usebibmacro{prenote}}
  {\usebibmacro{citeindex}%
   \printtext[bibhyperref]{\usebibmacro{cite}}}
  {\multicitedelim}
  {\usebibmacro{postnote}}

\DeclareCiteCommand{\textcite}
  {\boolfalse{cbx:parens}}
  {\usebibmacro{citeindex}%
   \printtext[bibhyperref]{\usebibmacro{textcite}}}
  {\ifbool{cbx:parens}
     {\bibcloseparen\global\boolfalse{cbx:parens}}
     {}%
   \multicitedelim}
  {\usebibmacro{textcite:postnote}}

\newbibmacro{string+doiurlisbn}[1]{%
  \iffieldundef{doi}{%
    \iffieldundef{url}{%
      \iffieldundef{isbn}{%
        \iffieldundef{issn}{%
          #1%
        }{%
          \href{http://books.google.com/books?vid=ISSN\thefield{issn}}{#1}%
        }%
      }{%
        \href{http://books.google.com/books?vid=ISBN\thefield{isbn}}{#1}%
      }%
    }{%
      \href{\thefield{url}}{#1}%
    }%
  }{%
    \href{https://doi.org/\thefield{doi}}{#1}%
  }%
}

\DeclareFieldFormat{title}{\usebibmacro{string+doiurlisbn}{\mkbibemph{#1}}}
\DeclareFieldFormat[article,inbook,incollection,inproceedings]{title}%
    {\usebibmacro{string+doiurlisbn}{#1}}
\input{includes/definitions.tex}%
\newcommand{\notationcommand}[3]{%
  \providecommand{#1}{}%
  \renewcommand{#1}{\newlink{#2}{#3}}%
}

\let\oldell\ell
\let\oldmu\mu
\let\oldeta\eta
\let\oldtau\tau
\let\oldGamma\Gamma
\let\oldDelta\Delta
\let\oldomega\omega
\let\oldrho\rho
\let\oldzeta\zeta
\let\oldpartial\partial
\let\olddelta\delta
\let\oldTheta\Theta
\let\oldOmega\Omega
\let\oldnu\nu

\notationcommand{\p}{def:p}{p}
\notationcommand{\q}{def:q}{q}
\notationcommand{\d}{def:dimension}{d}
\notationcommand{\RR}{def:radius}{R}
\notationcommand{\G}{def:lipschitz}{G}
\notationcommand{\T}{def:horizon}{T}
\notationcommand{\eps}{def:accuracy}{\varepsilon}
\notationcommand{\delta}{def:confidence}{\olddelta}

\notationcommand{\f}{def:objective}{f}
\NewDocumentCommand{\fstar}{o}{%
  \IfNoValueTF{#1}{\newlink{def:optimum}{f^{\star}}}{\newlink{def:optimum}{f^{\star}_{#1}}}%
}
\notationcommand{\g}{def:subgradient}{g}
\notationcommand{\X}{def:domain}{\mathcal{X}}
\notationcommand{\cG}{def:subgradients}{\mathcal{G}}
\notationcommand{\B}{def:ball}{B}
\newcommand{\ball}[2]{\newlink{def:ball}{B_{#1}^{#2}}}
\notationcommand{\ell}{def:norm}{\oldell}
\newcommand{\conj}{\newlink{def:conjugate}{\ast}}
\notationcommand{\partial}{def:subgradient}{\oldpartial}
\notationcommand{\R}{def:basic-notation}{\mathbb R}
\notationcommand{\E}{def:expectation}{\mathbb E}
\notationcommand{\Prb}{def:probability}{\mathbb P}
\notationcommand{\cN}{def:gaussian}{\mathcal N}
\notationcommand{\Z}{def:gaussian}{Z}
\notationcommand{\Id}{def:gaussian}{I}
\notationcommand{\Sphere}{def:sphere}{\mathbb S}
\notationcommand{\sphmu}{def:sphere-measure}{\oldmu}
\notationcommand{\pospart}{def:positive-part}{+}
\notationcommand{\1}{def:indicator}{\mathbf 1}
\notationcommand{\dist}{def:distance}{\operatorname{dist}}
\renewcommand{\norm}[1]{%
  \mathopen{\newlink{def:norm}{\left\lVert\vphantom{#1}\right.\kern-\nulldelimiterspace}}%
  #1%
  \mathclose{\newlink{def:norm}{\kern-\nulldelimiterspace\left.\vphantom{#1}\right\rVert}}%
}
\renewcommand{\ip}[2]{\newlink{def:inner-product}{\left\langle #1,#2\right\rangle}}

\notationcommand{\loss}{def:online-loss}{\oldell}
\notationcommand{\epssign}{def:online-signs}{\varepsilon}
\notationcommand{\Reg}{def:regret}{\operatorname{Reg}}
\notationcommand{\xonline}{def:online-iterates}{x}
\notationcommand{\xchase}{def:chasing-iterates}{x}
\notationcommand{\xquery}{def:query-points}{x}
\notationcommand{\ebasis}{def:standard-basis}{e}
\notationcommand{\Nphase}{def:phase-horizon}{N}
\notationcommand{\stail}{def:head-tail-exponent}{s}
\notationcommand{\K}{def:requests}{K}
\notationcommand{\z}{def:selector-points}{z}
\notationcommand{\zhat}{def:approximate-center}{\widehat z}
\notationcommand{\model}{def:bundle-model}{m}
\notationcommand{\lb}{def:lower-bound}{\oldell}
\notationcommand{\U}{def:upper-bound}{U}
\notationcommand{\Gamma}{def:phase-gap}{\oldGamma}
\NewDocumentCommand{\hminus}{o}{%
  \IfNoValueTF{#1}{\newlink{def:lower-level}{h_{-}}}{\newlink{def:lower-level}{h_{-}^{#1}}}%
}
\NewDocumentCommand{\hplus}{o}{%
  \IfNoValueTF{#1}{\newlink{def:upper-level}{h_{+}}}{\newlink{def:upper-level}{h_{+}^{#1}}}%
}
\DeclareRobustCommand{\serious}{\newlink{def:serious-step}{serious}}
\DeclareRobustCommand{\nullstep}{\newlink{def:null-step}{null}}
\notationcommand{\oracle}{def:oracle}{\mathfrak O}
\notationcommand{\selector}{def:selector-rule}{\mathsf C}
\NewDocumentCommand{\xbest}{o}{%
  \IfNoValueTF{#1}{\newlink{def:best-query}{x_{\mathrm{best}}}}{\newlink{def:best-query}{x_{\mathrm{best}}^{#1}}}%
}
\notationcommand{\xhat}{def:output}{\widehat x}
\notationcommand{\rho}{def:reduction-exponent}{\oldrho}

\notationcommand{\supporth}{def:support-function}{h}
\notationcommand{\st}{def:steiner}{\operatorname{st}}
\notationcommand{\omega}{def:mean-width}{\oldomega}
\notationcommand{\shat}{def:empirical-steiner}{\widehat s}
\notationcommand{\Nsample}{def:sample-count}{N}
\NewDocumentCommand{\errE}{o}{%
  \IfNoValueTF{#1}{\newlink{def:sampling-error}{E_{\mathrm{samp}}}}{\newlink{def:sampling-error}{E_{\mathrm{samp}}^{#1}}}%
}
\notationcommand{\tau}{def:threshold}{\oldtau}
\NewDocumentCommand{\Rhead}{o}{%
  \IfNoValueTF{#1}{\newlink{def:head-radius}{R_{1}}}{\newlink{def:head-radius}{R_{1}^{#1}}}%
}
\notationcommand{\Rtail}{def:tail-radius}{R}
\notationcommand{\rmin}{def:tail-exponent}{r}
\notationcommand{\snear}{def:head-exponent}{s}
\NewDocumentCommand{\Ld}{o}{%
  \IfNoValueTF{#1}{\newlink{def:log-dimension}{L_{\d}}}{\newlink{def:log-dimension}{L_{\d}^{#1}}}%
}
\notationcommand{\Lift}{def:lift}{\mathcal L}
\notationcommand{\mu}{def:mu}{\oldmu}
\notationcommand{\eta}{def:eta}{\oldeta}
\notationcommand{\gdir}{def:perturbation}{g}
\notationcommand{\usample}{def:sample-optimizer}{u}
\notationcommand{\vsample}{def:sample-optimizer}{v}
\notationcommand{\F}{def:sample-objective}{F}
\NewDocumentCommand{\Cgauss}{o}{%
  \IfNoValueTF{#1}{\newlink{def:gaussian-center}{\mathsf C_{\p,\q,\T}}}{\newlink{def:gaussian-center}{\mathsf C_{\p,\q,\T}^{#1}}}%
}
\NewDocumentCommand{\Phip}{o}{%
  \IfNoValueTF{#1}{\newlink{def:p-energy}{\Phi_{\p}}}{\newlink{def:p-energy}{\Phi_{\p}^{#1}}}%
}
\notationcommand{\D}{def:bregman}{D}
\NewDocumentCommand{\Cp}{o}{%
  \IfNoValueTF{#1}{\newlink{def:energy-center}{\mathsf C_{\p}}}{\newlink{def:energy-center}{\mathsf C_{\p}^{#1}}}%
}
\NewDocumentCommand{\Ppq}{o}{%
  \IfNoValueTF{#1}{\newlink{def:adapted-energy-center}{\mathsf P_{\p,\q}}}{\newlink{def:adapted-energy-center}{\mathsf P_{\p,\q}^{#1}}}%
}
\NewDocumentCommand{\Psipq}{o}{%
  \IfNoValueTF{#1}{\newlink{def:energy-regularizer}{\Psi_{\p,\q}}}{\newlink{def:energy-regularizer}{\Psi_{\p,\q}^{#1}}}%
}
\notationcommand{\Chat}{def:empirical-center}{\widehat C}
\notationcommand{\Y}{def:center-sample}{Y}
\notationcommand{\Yhat}{def:numerical-sample}{\widehat Y}
\NewDocumentCommand{\Esol}{o}{%
  \IfNoValueTF{#1}{\newlink{def:solver-error}{E_{\mathrm{sol}}}}{\newlink{def:solver-error}{E_{\mathrm{sol}}^{#1}}}%
}
\NewDocumentCommand{\Ectr}{o}{%
  \IfNoValueTF{#1}{\newlink{def:center-error}{E_{\mathrm{ctr}}}}{\newlink{def:center-error}{E_{\mathrm{ctr}}^{#1}}}%
}
\notationcommand{\mval}{def:sample-value}{m}
\notationcommand{\Delta}{def:value-increment}{\oldDelta}
\notationcommand{\ubar}{def:mean-components}{\bar u}
\notationcommand{\vbar}{def:mean-components}{\bar v}
\notationcommand{\zeta}{def:objective-tolerance}{\oldzeta}
\notationcommand{\cX}{def:function-sets}{\mathcal X}
\notationcommand{\cG}{def:function-sets}{\mathcal G}
\notationcommand{\cL}{def:linear-class}{\mathcal L}
\notationcommand{\cF}{def:function-class}{\mathcal F}
\notationcommand{\fat}{def:fat-shattering}{\operatorname{fat}}
\renewcommand{\sfat}[1]{\newlink{def:sequential-fat-shattering}{\operatorname{fat}^{\mathrm{seq}}_{#1}}}

\notationcommand{\bigO}{def:asymptotics}{O}
\notationcommand{\bigOtilde}{def:soft-asymptotics}{\widetilde O}
\notationcommand{\Theta}{def:asymptotics}{\oldTheta}
\notationcommand{\Omega}{def:asymptotics}{\oldOmega}
\notationcommand{\Thetatilde}{def:soft-asymptotics}{\widetilde\oldTheta}
\renewcommand{\bigo}[1]{\bigO(#1)}

\renewcommand{\bigopl}[2]{\bigO_{#1}\left(#2\right)}
\renewcommand{\bigotilde}[1]{\bigOtilde(#1)}
\renewcommand{\bigotildel}[1]{\bigOtilde\left(#1\right)}
\renewcommand{\bigotildep}[2]{\bigOtilde_{#1}(#2)}
\renewcommand{\bigotildepl}[2]{\bigOtilde_{#1}\left(#2\right)}

\notationcommand{\Compl}{def:randomized-complexity}{\operatorname{Compl}}
\notationcommand{\Oracle}{def:local-oracles}{\operatorname{Oracle}}
\notationcommand{\RandAlg}{def:randomized-algorithms}{\operatorname{RandAlg}}
\notationcommand{\hittime}{def:hitting-time}{T}
\hypersetup{colorlinks=true,allcolors=blue,linktoc=all,hypertexnames=false}

\ifusecolttemplate
  \newaliascnt{assumption}{theorem}
  
  \aliascntresetthe{assumption}
  \newaliascnt{fact}{theorem}
  \newtheorem{fact}[fact]{Fact}
  \aliascntresetthe{fact}
\else
  \theoremstyle{plain}
  \newtheorem{theorem}{Theorem}[section]
  \newcommand{\sharedtheorem}[2]{%
    \newaliascnt{#1}{theorem}%
    \newtheorem{#1}[#1]{#2}%
    \aliascntresetthe{#1}%
  }
  \sharedtheorem{lemma}{Lemma}
  \sharedtheorem{fact}{Fact}
  \sharedtheorem{proposition}{Proposition}
  \sharedtheorem{corollary}{Corollary}
  \theoremstyle{definition}
  \sharedtheorem{definition}{Definition}
  \sharedtheorem{remark}{Remark}
  \sharedtheorem{assumption}{Assumption}
  \newcommand{\coltauthor}[1]{\author{#1}}
  \newcommand{\AI}[1]{\section*{AI Statement}#1}
\fi

\crefname{algorithm}{Algorithm}{Algorithms}
\crefname{assumption}{Assumption}{Assumptions}
\crefname{fact}{Fact}{Facts}
\crefname{proposition}{Proposition}{Propositions}
\crefname{corollary}{Corollary}{Corollaries}
\crefname{lemma}{Lemma}{Lemmas}
\crefname{theorem}{Theorem}{Theorems}
\crefname{definition}{Definition}{Definitions}
\crefname{remark}{Remark}{Remarks}

\ifusecolttemplate
  \makeatletter
  \apptocmd{\@begintheorem}{\phantomsection}{}{}
  \apptocmd{\@opargbegintheorem}{\phantomsection}{}{}
  \makeatother
\fi

\title{Stable Movement for Nondual Lipschitz Convex Optimization: Efficiency and Nearly Optimal Oracle Rates}
\ifusecolttemplate
  \makeatletter
  \renewcommand{\@shorttitle}{Stable Movement for Nondual Lipschitz Convex Optimization}
  \makeatother
\fi

\coltauthor{
 \Name{David Martínez-Rubio}\Email{\href{mailto:david.martinezrubio@imdea.org}{david.martinezrubio@imdea.org}}\\
 \addr IMDEA Software Institute, Madrid, Spain
 \AND
 \Name{Cristóbal Guzmán}\Email{\href{mailto:crguzmanp@uc.cl}{crguzmanp@uc.cl}}\\
 \addr Institute for Mathematical and Computational Engineering, Faculty of Mathematics and
 School of Engineering, Pontificia Universidad Católica de Chile, Santiago, Chile
    }
\date{}

\newcommand\blfootnote[1]{%
\begingroup
\renewcommand\thefootnote{}\renewcommand\footnoteseptext{}\footnote{#1}%
\addtocounter{footnote}{-1}%
\endgroup
}

\begin{document}
\maketitle

\begin{abstract}
We study efficient algorithms for realizing the first-order oracle complexity of optimization of $\G$-Lipschitz convex functions with respect to the $\ell_{\q}$-norm over an $\ell_{\p}$-ball of radius $\RR$, where $1\leq \p,\q\leq \infty$.
For $\p<\q$, we obtain error
    $\bigOtilde_{\p,\q}(\G\RR/\T^{1/\p-(1/\q-1/2)_{\pospart}})$ after $\T$ oracle queries, attaining the lower bounds in a COLT 2015 open problem \citep{guzman2015open} with efficient algorithms. In particular, the rate is $\bigOtilde(\G\RR/\T)$ for Euclidean Lipschitzness with an $\ell_1$-ball domain of radius $\RR$ ($\p=1,\q=2$).
    Our solution consists of reducing convex Lipschitz optimization to the chasing nested convex sets problem \citep{bansal2020nested} in sublevel sets  of an evolving maximum of affine lower bounds, that is, the bundle \citep{lemarechal1995new}: at each query we either find a point with low function value or we produce a deep cut in the current sublevel of the bundle that we chase. The dichotomy between stability of selectors and forced movement by deep cuts bounds the number of iterations of the algorithm nearly optimally. 
For nested convex subsets of the $\ell_\p$ ball $\RR \ball{\p}{\d}$, we introduce novel stable selectors whose movement %
is nearly optimal in high dimensions.
A Monte Carlo average of the proposed selector achieves near-optimal rates with high probability and can be implemented for our optimization algorithm in polynomial time in the real-arithmetic model.
\end{abstract}

\blfootnote{\color{darkgray}Most non-local notation in this work links to its definition, using \href{https://damaru2.github.io/general/notations_with_links/}{this code}; for example, ${\newlink{def:lipschitz}{\color{darkgray}G}}$ links to its definition as the Lipschitz constant in $\norm{\cdot}_\q$ for the optimization problem we consider in this work.}

\clearpage
\setcounter{page}{1}

\section{Introduction}\label{sec:intro}

Lipschitz convex optimization is a classical and highly influential problem. Foundational algorithmic developments in this area include cutting-plane \citep{nemirovskiYudin1983,Khachiyan:1979}, subgradient \citep{Shor:1964,Shor:1985}, proximal \citep{Martinet:1970}, and mirror descent methods \citep{nemirovskiYudin1983}. Each of these developments has played a key role in areas including theoretical computer science \citep{groetschelLovaszSchrijver1993}, signal processing \citep{Beck:2009}, and machine learning \citep{Bottou:2018}, among others.

The standard framework in Lipschitz convex optimization accesses the objective $\f$ via a first-order oracle, which returns the function value $\f(x)$ and a subgradient $\g\in\partial \f(x)$ when queried at $x$. In high dimensions, virtually every prior algorithm is designed for the \emph{dual setting}: the feasible set $\X$ is bounded in a norm $\norm{\cdot}$ and the subgradients of $\f$ are bounded in the dual norm $\norm{\cdot}_\ast$, that is, $\f$ is Lipschitz with respect to the same norm that measures the domain. However, many problems are naturally \emph{nondual}, such as block MIR sampling \citep{BoyerWeissBigot2014} or convex relaxations of sparse PCA \citep{Dey:2017}. A prominent example is the minimization of a Euclidean-Lipschitz convex function over discrete probability distributions, which is a fundamental and widespread problem, for which the classical optimization approach would consist of mirror descent \citep{nemirovskiYudin1983} in the Euclidean geometry, converging at rate $\bigO(1/\sqrt{\T})$ after $\T$ queries: exactly the same rate as if the domain were the much larger Euclidean ball, failing to exploit the geometry of the smaller feasible set. A lower bound for algorithms using a local oracle yields $\Omega(1 / \T)$ rates for this problem \citep{guzman2015thesis,guzman2015open}.

Whether one can do better was posed, in greater generality, as a COLT 2015 open question \citep{guzman2015open}, motivated by this and more general lower bounds. The question asked for the high-dimensional oracle complexity of minimizing a convex function that is $\G$-Lipschitz in $\norm{\cdot}_{\q}$ over the ball $\RR\ball{\p}{\d}$, where $\p,\q\in[1,\infty]$, as well as its H\"older-smooth analogues \citep{guzmanNemirovski2015lower}. The open regime is $\p<\min\{\q,2\}$, as the rest of cases do get solved by mirror descent in the nonsmooth case.\footnote{The case $\p \geq 2$ when $\p < \q$ was also stated as an open question \citep{guzman2015open}. However, the lower bound is attained by an instance of mirror descent. Indeed, a $\G_{\q}$-Lipschitz function in $\ell_{\q}$ is also $\G_{\q}$-Lipschitz in $\ell_{\p}$. In fact the best Lipschitz constant $\G_{\p}$ in the $\ell_{\p}$ geometry can be $\G_{\p} \ll \G_{\q}$. Hence, for $\p \geq 2$, running mirror descent over $\RR \ball{\p}{\d}$ and using the $\ell_{\p}$ geometry gives a rate of $\bigo{\RR \G_{\p} / \T^{1 / \p}}$, which is no larger than the requested rate $\bigo{\RR \G_{\q} / \T^{1 / \p}}$.} In this regime,  \citet{bullins2026firstorder} recently showed the existence of strategies that nearly attain the rates of the lower bounds in the convex nonsmooth open regime $\p < \min\{\q, 2\}$, through a game-theoretic approach, with an algorithm with exponential runtime. These rates are
\begin{equation}\label{eq:intro-target-alt}
    \widetilde{O}\left(\frac{\G\RR}{\T^{\frac{1}{\p} -\left( \frac{1}{\q} - \frac{1}{2}\right)_{\pospart}}} \right)
\end{equation}
after $\T$ queries to the first-order oracle, where $\newtarget{def:positive-part}{(\cdot)_{\pospart}} = \max\{0, \cdot\}$. For $(\p,\q)=(1,2)$, this is $\Thetatilde(\G\RR/\T)$. The exponent is greater than $1 / 2$ in this regime while the previously fastest known rate, from mirror descent, in the most benign geometries is $\bigo{1 / \sqrt{\T}}$. Motivated by the positive results in the aforementioned paper, we study algorithms realizing the above rates in terms of first-order oracle complexity while running efficiently. Attaining \eqref{eq:intro-target-alt} seems to require breaking away from the regret-based paradigm that has dominated first-order methods since \citet{nemirovskiYudin1983},
see \cref{sec:further-related-work}. We instead exploit all past first-order information jointly, through a lower model of $\f$ built as the maximum of all observed linear lower bounds, that is, using a bundle method. 

We give a polynomial-time algorithm attaining the nearly optimal rates \eqref{eq:intro-target-alt}, thereby resolving the nonsmooth end of the COLT 2015 open question efficiently. The analysis of our bundle method hinges on a new connection: we reduce convex optimization to the online problem of chasing nested convex bodies \citep{bansal2020nested,argue2019nearly,sellke2020optimal}, so that the optimization oracle complexity is governed by the stability of a selector. We design new selectors whose movement is nearly optimal, in the high-dimensional regime $\T\le\d$, a setting that has received little attention in the online algorithms literature. 
Interestingly, our movement lower bounds establish a novel connection with the {\em sequential fat-shattering dimension} of linear functionals, a quantity that has been recently used to {\em upper bound} the oracle complexity for a wide spectrum of settings \citep{bullins2026firstorder}.

The chasing nested convex sets problem consists of an online game where at each round an adversary selects a convex set $\K_t$ and a player selects a point in it, $\z_t \in \K_t$, while the sets are nested $\K_{t+1} \subseteq \K_{t} \subseteq \cdots \subseteq \K_0$ and the initial point $\z_0 \in \K_0$ is fixed. The aim is usually to compare the movement cost of the player in some norm after $\T$ rounds, $\sum_{t=0}^{\T-1} \norm{\z_t - \z_{t+1}}$, relative to the movement of an offline strategy that knows $\K_{\T}$ in advance. For our purposes, we will actually look at solely bounding the movement cost $\sum_{t=0}^{\T-1} \norm{\z_t - \z_{t+1}}$ in absolute terms. For instance, \citet{bubeck2020nested} show that for $\ell_2$-norm, one can provide a strategy with movement bounded by $\bigo{\omega(\K_0)\sqrt{\d \log(\T / \omega(\K_0))}}$, where $\omega(\K_0)$ is the mean width of $\K_0$ and $\d$ is the dimension, which results in a $\bigOtilde(1)$ near-optimal movement in the unit $\ell_1$ ball. 

\begin{figure}[h!]
  \centering
  \begin{tikzpicture}[
      x=1cm,
      y=.85cm,
      >=stealth,
      every node/.style={font=\small}
    ]
    \draw[->, black!65] (0.65,0.35) -- (11.45,0.35);
    \draw[->, black!65] (0.8,0.2) -- (0.8,5.65);
    \draw[densely dashed, black!45] (0.8,2.25) -- (11.3,2.25);
    \draw[densely dashed, black!45] (0.8,4.0) -- (11.3,4.0);
    \node[anchor=east] at (0.7,2.25) {$\hminus$};
    \node[anchor=east] at (0.7,4.0) {$\hplus$};

    \draw[densely dotted, black!50] (0.9,5.36) -- (5.05,0.42);
    \draw[densely dotted, black!50] (1.0,1.24) -- (11.0,2.14);
    \draw[densely dotted, blue!55!black] (5.1,0.42) -- (10.6,5.48);
    \draw[very thick, black!75] (0.9,5.36) -- (4.1,1.55);
    \draw[very thick, black!75] (4.1,1.55) -- (6.6,1.80);
    \draw[very thick, blue!65!black] (6.6,1.80) -- (10.6,5.48);
    \node[black!65, anchor=east] at (1.55,4.60) {$a_1$};
    \node[black!65, anchor=south] at (2.15,1.28) {$a_2$};
    \node[black, anchor=west] at (10.25,5.05) {$a_3$};

    \draw[gray!65, line width=2.6pt] (3.51,2.25) -- (11.05,2.25);
    \draw[black, line width=3.2pt] (3.51,2.25) -- (7.09,2.25);
    \node[anchor=south] at (5.30,2.30) {$\K_{3}$};
    \node[text=gray!70, anchor=north] at (8.75,1.70) {old $K_2$};

    \coordinate (querypoint) at (9.70,4.65);
    \fill[blue!65!black] (querypoint) circle (2pt);
    \node[black, anchor=east] at (9.48,4.72)
      {$\f(\z_2)$};
    \draw[densely dotted, black!55] (9.70,0.35) -- (querypoint);
    \node[anchor=north] at (9.70,0.31) {$\z_2$};

    \draw[densely dotted, black!55] (7.09,2.25) -- (7.09,2.82);
    \draw[<->, semithick] (7.12,2.72) -- (9.67,2.72)
      node[midway, fill=white, inner sep=1.5pt] {$\ge\Gamma/(2\G)$};
  \end{tikzpicture}
    \caption{A one-dimensional section of a step where the function value is greater than $h_+$ (null step).  The old model
    $\max\{a_1,a_2\}$ lies below $\hminus$ at $\z_2$ (that is, $\z_2 \in \K_2$), whereas the new tangent
  satisfies $a_3(\z_2)=\f(\z_2)>\hplus$.  Adding $a_3$ reduces the
    $(\hminus)$-sublevel to the thick black segment ($\K_{3} \subset \K_{2}$), separated from
  $\z_2$ by at least $\Gamma/(2\G)$ as the subgradient norm is no larger than $\G$, where $\Gamma$ is the duality gap at the beginning of a phase and $h_+ - h_- = \Gamma / 2$.}
  \label{fig:bundle-intuition}
\end{figure}
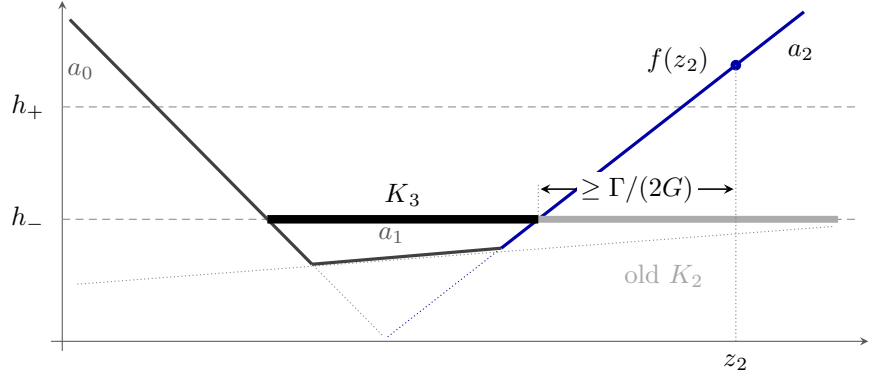

A bundle method by \citet{lemarechal1995new}, designed for the dual case of a Lipschitz convex problem in $\norm{\cdot}_2$ over $\ball{2}{\d}$, keeps a sublevel set of a bundle and projects the previous iterate onto the new sublevel set after a query, in phases that decrease the duality gap by a constant factor, where phases terminate due to a bounded quadratic potential.
Our optimization algorithm is inspired by this bundle method, but we make a direct connection to the problem of chasing  nested convex bodies and reduce the optimization problem to it. The final idea is simple and elegant. For a convex function that is $\G$-Lipschitz in $\norm{\cdot}_\q$, after $t$ queries $(\xquery_i)_{i\leq t}$, we have received oracle answers $(\f(\xquery_i),\g_i)_{i \leq t}$ where the subgradient $\g_i\in\partial \f(\xquery_i)$ satisfies $\norm{\g_i}_{\q^\conj} \leq \G$ and $\f(\xquery_i) + \ip{\g_i}{x-\xquery_i} \leq \f(x)$. Let the {\em bundle model} and its {\em lower/upper bounds} be
\[
  \model_t(x)=\max_{i\le t}\{\f(\xquery_i)+\ip{\g_i}{x-\xquery_i}\},
  \qquad
\lb_t\le\min_{x\in\X}\model_t(x),
  \qquad
  \U_t=\min_{i\le t}\f(\xquery_i).
\]
Here $\lb_t$ is a certified lower bound maintained by our \cref{alg:lipschitz-bundle}.
Convexity gives $\lb_t\le \min_{x\in\X} \f(x)\le\U_t$. We run the algorithm in phases. Freeze
the gap $\Gamma=\U-\lb$ and two levels
$\hminus=\lb+\Gamma/4<\hplus=\lb+3\Gamma/4$, and query a stable selector
$\z_t\in\K_t:=\{x\in\X:\model_t(x)\le\hminus\}$  for the chasing nested convex sets problem. If
$\f(\z_t)\le\hplus$, the upper bound improves by a constant fraction and we move to the next phase.  Otherwise
the new query $\xquery_{t+1}=\z_t$ gives the tangent
$a_{t+1}(x)=\f(\z_t)+\ip{\g_{t+1}}{x-\z_t}$ and
$\K_{t+1}=\K_t\cap\{a_{t+1}\le\hminus\}$.  Since
$\norm{\g_{t+1}}_{\q^{\conj}}\le\G$, this newly included halfspace is at $\ell_{\q}$-distance
at least $(\hplus-\hminus)/\G=\Gamma/(2\G)$ from $\z_t$.  Thus every such
\emph{deep cut} forces movement of the selector, unless the new body is empty;
see \cref{fig:bundle-intuition}.
An upper bound on the total movement $\sum_{t=0}^{\T-1}\norm{\z_{t+1}-\z_t}_{\q}$ after $\T$ steps coming from the chasing convex bodies problem, along with the fact that the movement is lower bounded by $\T \cdot \Gamma/(2\G)$ gives a maximum number of steps until the certified lower bound increases by at least $\Gamma/8$, which can occur before the chased sublevel set becomes empty, or we have $\f(\z_{t}) \leq \hplus$. Either way, we decrease the gap by a constant factor and continue to the next phase.

Since certifying emptiness of a set can be computationally challenging, in the actual algorithm, we can in fact stop before seeing an empty sublevel set, if our certified lower bound increases enough, cf. Line \ref{line:buffered-bundle-emptiness} of \cref{alg:lipschitz-bundle}.

\paragraph{Contributions.}
The technical results of this work naturally divide into three modular pieces.
\begin{enumerate}[leftmargin=2.2em]
  \item For chasing nested convex subsets of $\ball{\p}{\d}$, with movement measured in the $\ell_{\q}$-norm, we prove that the optimal
  movement for the chasing nested convex bodies problem over $\T$ requests, up to logarithmic factors in $\d$ and $\T$, is
        $\Thetatilde_{\p,\q}\Big(\T^{1-\frac1\p+\left(\frac1\q-\frac12\right)_{\pospart}}\Big)$ for $\p<\q$ and $\Thetatilde_{\p,\q}(\d^{1 / \q - 1 / \p} \T^{1 - 1 / \max\{\q, 2\}})$ for $\p \geq \q$, both in the
  high-dimensional regime \(\T\leq \d\), cf. \cref{thm:movement-main}. %
For the novel optimization regime $\p < \min\{\q, 2\}$, we leverage a head-tail decomposition of the $\ell_\p$-ball, cf. \Cref{lem:head-tail}, and define a novel {\em mixed-norm Gaussian selector} by substituting the support function of the Steiner point, that was nearly optimal for $\p=1,\q=2$, by a proximal operator of a mixed-norm in a lifted space designed from the head-tail decomposition, cf. \eqref{eq:unified-sample-program} and \eqref{eq:gaussian-center}.
    \item We reduce Lipschitz convex optimization to nested convex-body chasing in \cref{thm:movement-reduction}. This reduction can also accommodate for inexact versions of stable selectors, cf. \cref{lem:buffered-implementation}, which are crucial for efficient implementations. The reduction follows essentially from the Stable Movement Bundle Method, \Cref{alg:lipschitz-bundle}.

    \item The exact mixed-norm Gaussian selector requires high-dimensional integration, but we also provide a Monte Carlo approximation that runs in polynomial time in the real-arithmetic model for  \cref{alg:lipschitz-bundle}, where sets are $\ball{\p}{\d}$ intersected with halfspaces. The efficiency follows from the polynomial running time of the Ellipsoid Method for the model level subproblems, ensuring inner and outer balls in each of the requested sets, and the proximal optimization problems behind the mixed-norm Gaussian selector.
\end{enumerate}
The combination of these three techniques leads to our main result (\cref{thm:optimization-main}): an efficient
    algorithm attaining the oracle rates of \citep{bullins2026firstorder} and
    resolving the nonsmooth $\ell_{\p}/\ell_{\q}$ case of the COLT 2015 open question \citep{guzman2015open} efficiently.\footnote{Note that the lower bounds in the open question are only applicable to deterministic algorithms \citep{guzman2015thesis,guzman2015open}, while our efficient algorithm is randomized. In \Cref{thm:LB_randomized_pq} we provide analogous lower bounds that hold against randomized algorithms.}

\subsection{Further related work}\label{sec:further-related-work}
\paragraph{Failure of online learning and online-to-batch}
In online convex optimization, the mirror descent algorithm  updates the current point 
using a subgradient of the observed loss and a Bregman penalty for moving
away from that point \citep{nemirovskiYudin1983,Beck:2003}.  With a suitable regularizer, it achieves nearly minimax
regret for a broad class of convex online learning problems
\citep{srebroSridharanTewari2011}.  This universality does not imply that
ordinary regret bounds recover every optimal rate for minimizing 
convex functions in the offline setting.
The obstruction is already one-dimensional.  It concerns reductions through
ordinary cumulative regret uniformly over adversarial linear losses; it does
not rule out every online-inspired construction.  Indeed, the classical
Rademacher lower bound already applies to the one-coordinate linear losses
$\newtarget{def:online-loss}{\loss_t}(x)=\epssign_t x_1$ on $ \ball{\p}{\d}$, where the signs $\newtarget{def:online-signs}{\epssign_t}$ are
independent and uniform on $\{-1,1\}$, independently of the learner's
randomness.  For every possibly randomized
learner choosing $\newtarget{def:online-iterates}{\xonline_t}$, ordinary cumulative regret satisfies
\begin{equation*}
  \newtarget{def:regret}{\Reg_{\T}}
  \defi
  \sum_{t=1}^{\T}\loss_t(\xonline_t)
  -\min_{u\in \ball{\p}{\d}}\sum_{t=1}^{\T}\loss_t(u),
  \qquad
  \E\Reg_{\T}
  \geq
  \,\E\abs{\sum_{t=1}^{\T}\epssign_t}
  =\Omega(\sqrt \T);
\end{equation*}
see, for example, \citet[Section~7.1]{abernethy2009stochastic} and
\citet{cesaBianchiLugosi2006}.  Consequently, a standard adversarial-regret
bound followed by the ordinary online-to-batch conversion cannot certify
$o(1/\sqrt \T)$. We note that the worst-case complexity of mirror descent for optimization does not necessarily have to follow that of the worst-case example in the adversarial case along with online-to-batch conversion, but this argument provides some evidence that the algorithm may not be able to improve over $\bigO(1 / \sqrt{\T})$ under the general assumptions considered in this work. We also note that precisely the open regimes in \citep{guzman2015open} are the ones where the convergence would be faster than $\bigO(1 / \sqrt{\T})$, including the fastest rate of $\bigOtilde(1 / \T)$ when $\p=1$, $\q \geq 2$.

This classical barrier allows adversarial losses that need not be
affine tangent lower bounds of one fixed convex objective. As explained, we instead use the maximum of all affine lower bounds given by first-order oracle responses to construct a single lower model of the objective.

\paragraph{Chasing convex sets.} The problem of chasing convex sets was introduced by \citet{friedmanLinial1993}.  Fix a norm $\norm{\cdot}$ on $\R^{\d}$. As explained, in this problem the player starts at $\newtarget{def:chasing-iterates}{\z_0}$, sees a convex request $\K_t$, and then chooses $\z_t\in \K_t$, paying movement $\sum_{t=0}^{\T-1}\norm{\z_{t+1}-\z_t}$.  For nested requests, after $\T$ iterations, the competitive benchmark
is the $\|\cdot\|$-distance from $\z_0$ to $\K_{\T}$: this value is achieved by satisfying every request by moving to the last body and staying there. %
An algorithm is $C$-competitive if its movement is at most $C$ times this clairvoyant optimum.  For the nested problem, where $\K_{t} \subset \K_{t-1}$ for all $t$, \citet{bansal2020nested} first proved a finite
dimension-dependent competitive ratio, quantitatively
$6^{\d}(\d!)^2$.  \citet{argue2019nearly} improved this to $\bigO(\d\log \d)$ in every
normed space, nearly matching the $\Omega(\d)$ lower bound in $\ell_\infty^{\d}$.
In Euclidean space, \citet{bubeck2020nested} obtained
$\bigO(\sqrt{\d\log \d})$ and showed that the memoryless rule which follows the
Steiner point has the horizon-sensitive bound
$\bigO(\min\{\d,\sqrt{\d\log \T}\})$, where $\T$ is the number of requests.  The
Steiner point is a classical selector belonging to the requested body; its
support-function representation makes it stable under Hausdorff perturbations,
and it was known to have the optimal Euclidean Hausdorff-Lipschitz
constant and the property that the Steiner point of a Minkowski sum is the sum
of the Steiner points.  Building on this idea, \citet{sellke2020optimal}
introduced a functional Steiner point and a corresponding potential, extending
the approach to non-nested convex-body chasing.  They obtained competitive
ratio $\d$ for convex bodies and $\d+1$ for convex functions in an arbitrary
$\d$-dimensional normed space, with an $\bigO(\sqrt{\d\log \T})$ Euclidean refinement.
The classical $\Omega(\d)$ lower bound in $\ell_\infty^{\d}$ establishes optimality
when $\T\ge \d$. We however focus on the high-dimensional setting where $\T\leq\d$.
Our reduction for optimization, \cref{thm:movement-reduction}, requires an
absolute movement bound, regardless of the offline comparator's cost.  The distinction between movement bounds and competitive ratio matters even for short horizons: \cref{prop:nondual-movement-lower} gives a competitive ratio lower bound $\Omega(\sqrt{\T})$ for Euclidean movement in $\ball{\p}{\d}$, $1\le\p\le2$, in high dimension $\d \geq \T$, whereas the movement can be of lower order, and as small as $\bigotilde{1}$ when $\p=1$ and $\q \geq 2$.

When movement is charged in the same norm $\norm{\cdot}$ in which the initial request is bounded, the two notions essentially coincide: up to universal constants, an $F(\d,\T)$-competitive algorithm for nested chasing is the same as an algorithm whose movement is $\bigO(F(\d,\T)\,r)$ whenever $\K_1\subseteq\z_0+r\B_{\norm{\cdot}}$, where $\B_{\norm{\cdot}}$ is the unit ball of $\norm{\cdot}$ \citep[Claim~1.4]{argue2019nearly}. In our setting the requests lie in $\RR\ball{\p}{\d}$ but movement is charged in $\ell_{\q}$ with $\p<\q$, and the two notions diverge. %

The classical Steiner point has been used as an important selector for chasing convex bodies, and it consists of averages of support points over directions.
Our proposed selector, in \cref{eq:gaussian-center}, instead averages
minimizers of a different regularized objective. The regularization allows movement bounds
adapted to $\RR \ball{\p}{\d}$, the $\q$-norm and the horizon $\T$, avoiding the $\operatorname{poly}(\d)$ dependence in the classical selectors like the Steiner point when $\T\le \d$.

\paragraph{Concurrent works.} The manuscript  \citep{ouyang2026algorithm} appeared on arXiv 8 days before ours as we prepared the writing of our results. The work studies the specific case of quadratics when $(\p, \q)=(1,2)$, linked to the smooth end of the open problem in \citep{guzman2015open}. It contains an idea that we had also independently used, about splitting coordinates into a head (their absolute value is large enough) and a tail (the rest), see \cref{lem:head-tail,prop:low-p-center}. We and coauthors resolved the full open question efficiently for the general Hölder smooth case in \citep{martinezrubio2026nearoptimalacceleration} using the results of the present paper. \citet{loginov2026complexity} appeared on arXiv 9 days after ours and solves the specific case $\p=1$ and $\q=2$ using similar ideas, such as making use of the Steiner point. Our solution instead solves the general problem for all $\p,\q$ and establishes general links between optimization and chasing nested convex sets, as well as between sequential fat-shattering and movement in the latter problem.

\section{Setting and main results}

\label{sec:setting-main-results}

Throughout, we work on the vector space $\newtarget{def:basic-notation}{\R}^{\newtarget{def:dimension}{\d}}$. We denote by $\newtarget{def:standard-basis}{\ebasis_i}$ the $i$-th standard basis vector.  
The parameters $\newtarget{def:p}{\p},\newtarget{def:q}{\q}\in[1,\infty]$
specify the domain $\newtarget{def:domain}{\X}=\newtarget{def:radius}{\RR}\ball{\p}{\d}$ and the $\norm{\cdot}_\q$ norm used for the objective's Lipschitzness.  We write $\newtarget{def:objective}{\f}$ for the convex objective and
$\newtarget{def:optimum}{\fstar}=\min_{x\in\X}\f(x)$ for its constrained minimum.
For $r\in[1,\infty]$, we use the convention $1/\infty=0$ and let $\newtarget{def:conjugate}{r^{\conj}}$
be its conjugate exponent, so that $1/r+1/r^{\conj}=1$. Let the norm
$\newtarget{def:norm}{\norm{x}_r}=(\sum_i\abs{x_i}^r)^{1/r}$ for finite $r$, and
$\norm{x}_\infty=\max_i\abs{x_i}$; $\newtarget{def:inner-product}{\ip{x}{y}}=\sum_i x_i y_i$. Write
\(  \newtarget{def:ball}{\ball{r}{\d}}=\set{x\in\R^{\d}:\norm{x}_r\le1},\)
and $\ball{r}{\d}(a,\RR)=a+\RR\ball{r}{\d}$.
We use $\newtarget{def:indicator}{\1}_A$ for the indicator of an event $A$ and
$\newtarget{def:distance}{\dist}_r(x,S)=\inf_{y\in S}\norm{x-y}_r$, with $\inf\varnothing=\infty$.
We write $\newtarget{def:probability}{\Prb}$ and $\newtarget{def:expectation}{\E}$ for probability and expectation, respectively.
A first-order oracle for a convex $\newtarget{def:lipschitz}{\G}$-Lipschitz function in $\ell_{\q}$ returns
$\f(x)$ and a subgradient $\newtarget{def:subgradient}{\g}\in\partial \f(x)$ with $\norm{\g}_{\q^{\conj}}\le \G$;
thus $\f(y)\ge\f(x)+\ip{\g}{y-x}$ for every $y\in\X$.
Throughout, $\newtarget{def:asymptotics}{\bigO_{\p,\q}}$ allows constants depending on fixed norm parameters;
$\newtarget{def:soft-asymptotics}{\bigOtilde_{\p,\q}}$ additionally omits factors polynomial in logarithms
of $\d,\T$, and of the inverse failure probability $\newtarget{def:confidence}{\delta}\in(0,1)$ and
target error $\newtarget{def:accuracy}{\eps}>0$.
Similarly for $\Omega(\cdot)$ and $\Theta(\cdot)$ notations.
Unless indicated otherwise, $\newtarget{def:gaussian}{\Z}\sim\cN(0,\Id_{\d})$
is a standard Gaussian vector, where $\Id_{\d}$ is the identity matrix.

Formally, the nested convex body chasing in $\RR \ball{\p}{\d}$ is the following problem.  For a
horizon $\newtarget{def:horizon}{\T}$, nonempty compact convex requests arrive in a nested sequence
\(
  \RR \ball{\p}{\d}\supseteq \newtarget{def:requests}{\K_0}\supseteq \K_1\supseteq\cdots\supseteq \K_{\T},
\)
and a selector chooses $\newtarget{def:selector-points}{\z_t}\in \K_t$ for every $0\le t\le \T$.  We measure its
performance by the total $\ell_{\q}$-movement
$\sum_{t=0}^{\T-1}\norm{\z_{t+1}-\z_t}_{\q}$ and seek bounds that hold uniformly
over every such sequence. 
Our main results are summarized in the following.

\begin{proposition}[Movement and competitive ratio lower bounds]
\label{prop:nondual-movement-lower}
\linktoproof{prop:nondual-movement-lower}
Fix $1\le\p,\q\le\infty$, $\T\le\d$, and $\RR>0$.
Starting from $\z_0=0$, an adaptive adversary can
generate compact convex requests
$\RR\ball{\p}{\d}=\K_0\supseteq\K_1\supseteq\cdots\supseteq\K_{\T} \neq \emptyset$
that force every possibly randomized online selector to incur a competitive ratio $\Omega_{\p, \q}(\T^{1-1/\max\{\q, 2\}})$ in $\norm{\cdot}_\q$ and movement
\[
  \sum_{t=0}^{\T-1}\norm{\z_{t+1}-\z_t}_{\q}
  \gtrsim_{\p,\q}
  \begin{cases}
    \RR\T^{1-1/\p+(1/\q-1/2)_{\pospart}}, & \p\le\q,\\
      {\RR\d^{1/\q-1/\p}\T^{1-1/\max\{\q,2\}}}, & \p>\q.
  \end{cases}
\]
\end{proposition}
Note that although the competitive ratio above is at least $\Omega_{\p, \q}(\sqrt{\T})$, the movement bound can be lower. We only require absolute movement bounds for our optimization results. In \cref{sec:fat_shattering_and_lower_bounds}, we show a more general result involving more general geometries and lower bounds depending on the \emph{sequential fat-shattering} dimension of relevant linear function classes. Now we state our main result about attainable movement bounds.

\begin{theorem}[Movement over $\RR \ball{\p}{\d}$ in $\ell_\q$]
\label{thm:movement-main}\linktoproof{thm:movement-main}
Fix $1\le\p,\q\le\infty$ and $\T\ge2$.
There is a deterministic selector which, for every nested sequence in
$\RR\ball{\p}{\d}$, satisfies
\begin{equation}\label{eq:movement-main}
  \sum_{t=0}^{\T-1}\norm{\z_{t+1}-\z_t}_{\q}
  =\bigotildep{\p,\q}{\RR}
  \begin{cases}
    \T^{1-1/\p+(1/\q-1/2)_{\pospart}}, & \p\le\q,\\[1mm]
    \d^{1/\q-1/\p}\T^{1-1/\max\{\q,2\}}, & \p>\q.
  \end{cases}
\end{equation}
For $\p< \min\{\q, 2\}$, a Monte Carlo approximation achieves the same movement bound up to logarithmic factors with high probability, which we can implement in polynomial time for requests generated by \cref{alg:lipschitz-bundle}. 
For $p\geq \min\{q,2\}$ the selector can be computed in polynomial time by a deterministic algorithm.
\end{theorem}
  We provide two selectors achieving the rates above: the mixed-norm Gaussian center in \cref{eq:gaussian-center} for $\p<\min\{\q,2\}$, and the potential-energy selector in \cref{eq:energy-complement-selector} for $\p\ge\min\{\q,2\}$; see \cref{prop:low-p-center,prop:energy-complement}. The latter consists of a simple projection.

For ${\p\le\min\{\q,2\}, \q\in(1,\infty]}$ and bodies generated by \cref{alg:lipschitz-bundle},
\cref{thm:polynomial-implementation} gives polynomial-time feasible
approximations, in the real-arithmetic model, with a degree that may depend on $\p$ and $\q$, cf. \cref{thm:polynomial-implementation}, that are enough for our optimization problem.
Our algorithm with the deterministic selector matches the high-dimensional lower bounds of \citet{guzman2015thesis,guzman2015open}, up to logarithmic factors. The next theorem states our results for the efficient randomized algorithm. In \cref{thm:LB-randomized-budget}, we show that even for randomized algorithms, the lower bounds stay essentially the same, proving the optimality of the following theorem in terms of first-order oracle complexity.
\begin{theorem}[Nonsmooth $\ell_{\p}/\ell_{\q}$ optimization]
\label{thm:optimization-main}\linktoproof{thm:optimization-main}
    Fix ${1\le\p\le\min\{\q,2\}}$ with $\q \in (1, \infty]$.
    Let $\f:\R^{\d} \to\R$ be convex and $\G$-Lipschitz in $\norm{\cdot}_{\q}$, $\T\ge2$, and fix
$\delta\in(0,1)$. There is a polynomial-time algorithm in the real-arithmetic model, using at most $\T$ first-order queries returning a
feasible point $\newtarget{def:output}{\xhat_{\T}}$ such that with probability at least $1-\delta$
\begin{equation}\label{eq:optimization-main}
  \f(\xhat_{\T})-\min_{x\in \RR \ball{\p}{\d}}\f(x)
  =
  \bigOtilde_{\p,\q}\!\left(
    \frac{\G\RR}{\T^{1/\p-(1/\q-1/2)_{\pospart}}}
  \right).
\end{equation}
\end{theorem}

As mentioned in the introduction, the cases $\p\ge\q$ and $2\le\p\le\q$ are already solved
optimally up to logarithmic factors by mirror descent %
\citep{nemirovskiYudin1983}. %
The rest of the paper therefore focuses mostly on $\p<\min\{\q, 2\}$,
which in particular implies $\q>1$. 

\section{The modular two-level bundle algorithm}
\label{sec:bundle}

Now we formally introduce and analyze the algorithm explained in \cref{sec:intro}, specializing to the $\ell_\p$/$\ell_\q$ case for simplicity, although our arguments work in general norms. After queries $\newtarget{def:query-points}{\xquery_i}$ with subgradients $\g_i$, define the bundle model
\(
  \newtarget{def:bundle-model}{\model_t}(x)
  \defi
  \max_{i\le t}\set{\f(\xquery_i)+\ip{\g_i}{x-\xquery_i}}.
\)
Recall that convexity gives $\model_t\le \f$ on $\X=\RR \ball{\p}{\d}$.  Hence
\(
  \newtarget{def:lower-bound}{\lb_t}\le\min_{x\in\X}\model_t(x)
  \le \fstar,
  \newtarget{def:upper-bound}{\U_t}\defi\min_{i\le t}\f(\xquery_i)\ge \fstar.
\)
Here $\lb_t$ is the certified lower bound maintained by \cref{alg:lipschitz-bundle}.
The algorithm fixes $\lb$ until the stopping test in line~\ref{line:buffered-bundle-emptiness} certifies an improvement.

\begin{algorithm}[h!]
\caption{Stable Movement Bundle Method}
\label{alg:lipschitz-bundle}
\begin{algorithmic}[1]
    \REQUIRE $\X=\RR \ball{\p}{\d}$, first-order oracle $\oracle$ for a convex $\f$ $\G$-Lipschitz in $\norm{\cdot}_\q$, $\p \leq \q$, target $\eps$, selector $\selector$, initial point $\xquery_0\in\X$. Initial lower bound  $\lb\gets \f(\xquery_0)-2\G\RR$
    \STATE\label{line:buffered-bundle-initialization} $(\f(\xquery_0), \g) \gets \oracle(\xquery_0)$; initialize $\model(x) = \f(\xquery_0) + \ip{\g}{x-\xquery_0}$,
    $\U\gets \f(\xquery_0)$, $\xbest\gets \xquery_0$
  \WHILE{$\U-\lb>\eps$}
    \STATE\label{line:buffered-bundle-setup-while} $\Gamma\gets \U-\lb$;
      $\hminus\gets\lb+\Gamma/4$;
      $\hplus\gets\lb+3\Gamma/4$; \(\lb_0\gets\lb\)
    \LOOP
      \STATE $\K\gets\set{x\in\X:\model(x)\le \hminus}$ \Comment{Current model sublevel set}
      \STATE Get $a,b\in \R$, $y\in \X$ such that $a\leq \min_{x\in \X}\model(x) \leq \model(y)\leq b$, and $b-a\leq \Gamma/32$\label{line:buffered-certified-near-minimum-model}
      \STATE\textbf{if} $a\geq  \lb_0+\Gamma/8$ \textbf{then} $\lb\gets \max\{\lb,a\}$;  \textbf{break} \textbf{endif} \Comment{End of phase by LB improvement}\label{line:buffered-bundle-emptiness}
      \STATE $\z\gets\selector(\K)$ and $(\f(\z),\g) \gets \oracle(\z)$
      \STATE \textbf{if} $\f(\z)<\U$ \textbf{then} $\xbest\gets \z$; $\U \gets\f(\z)$ \textbf{endif} 
      \STATE\label{line:buffered-bundle-model-update} $\model(x)\gets
          \max\{\model(x),\f(\z)+\ip{\g}{x-\z}\}$
      \STATE \textbf{if} $\f(\z)\le \hplus$ \textbf{then} %
      \textbf{break} \textbf{endif} \Comment{End of phase by serious step}\label{line:buffered-bundle-serious-update} 
    \ENDLOOP
  \ENDWHILE
  \STATE \textbf{return} $\xbest$
\end{algorithmic}
\end{algorithm}

The algorithm proceeds in phases.  During each phase it keeps two fixed
levels between the current lower and upper bounds and queries a selector on a
nested sequence of model sublevel sets.  The phase ends when either bound
improves by a constant fraction of the current gap, after which the levels
and the selector are restarted.
At the start of a phase, freeze
\begin{equation}\label{eq:bundle-levels}
  \newtarget{def:phase-gap}{\Gamma}=\U-\lb,
  \qquad
  \newtarget{def:lower-level}{\hminus}=\lb+\frac14\Gamma,
  \qquad
  \newtarget{def:upper-level}{\hplus}=\lb+\frac34\Gamma.
\end{equation}
The lower level defines the body to be chased; the upper level decides
whether a query is a \serious{} step.  A query at $\z$ is a
\newtarget{def:serious-step}{\emph{\serious{} step}} if $\f(\z)\le\hplus$, and a
\newtarget{def:null-step}{\emph{\nullstep{} step}} otherwise.
Both add the returned affine lower bound to the bundle.  A \serious{} step
ends the phase; a \nullstep{} step continues it unless the stopping test in line~\ref{line:buffered-bundle-emptiness} certifies a lower-bound improvement of at least $\Gamma/8$. The fractions $1/4$ and $3/4$ are chosen
only for simplicity; any fixed constants $0<\beta<\alpha<1$ also work.

In the algorithm, $\newtarget{def:oracle}{\oracle}(x)=(\f(x),\g)$ is a first-order
oracle, $\newtarget{def:selector-rule}{\selector}(\K)\in\K$ is the chosen selector,
and $\newtarget{def:best-query}{\xbest}$ stores a queried point attaining $\U$.

\begin{lemma}[A \nullstep{} step is a deep cut]
\label{lem:null-deep}\linktoproof{lem:null-deep}
During a phase of \cref{alg:lipschitz-bundle}, with gap $\Gamma$ and
levels $\hminus,\hplus$ from \cref{eq:bundle-levels}, suppose the query at $\z$ returns $\g$ and
$\f(\z)>\hplus$.  If $\g\ne0$, then
\begin{equation*}
  \dist_{\q}\bigl(\z,\set{x:\f(\z)+\ip{\g}{x-\z}\le \hminus}\bigr)
  =\frac{\f(\z)-\hminus}{\norm{\g}_{\q^{\conj}}}
  \ge\frac{\Gamma}{2\G}.
\end{equation*}
\end{lemma}

If the next body to chase is nonempty, comparing the movement caused by a null step with total movement bound limits the number of \nullstep{} steps in each phase.

\begin{proposition}[Movement-to-optimization reduction]
\label{thm:movement-reduction}\linktoproof{thm:movement-reduction}
Let ${\p\le\q}$ and $0<\newtarget{def:reduction-exponent}{\rho}\le1$, and suppose we use a selector $\selector$ for chasing nested convex sets in
$\RR \ball{\p}{\d}$ that satisfies, for every $\newtarget{def:phase-horizon}{\Nphase}$,
\begin{equation} \label{eq:center-movement}
  \sum_{t=0}^{\Nphase-1}\norm{\z_{t+1}-\z_t}_{\q}
  =\bigOtilde\!\left(\RR \Nphase^{1-\rho}\right).
\end{equation}
    Then, \cref{alg:lipschitz-bundle} obtains an $\eps$-minimizer of a convex $\G$-Lipschitz function in $\norm{\cdot}_\q$ using $\T$ first-order calls, where
\begin{equation}\label{eq:oracle-compl-from-movement}
    \T = \bigotildel{1+\left(\frac{\G\RR}{\eps}\right)^{1/\rho}}.
\end{equation} 
\end{proposition}

We now show that we can achieve the same first-order oracle complexity if we have a good approximation to a sequence with movement bounded as in \cref{eq:center-movement}. The optimization problem allows for errors that are absorbed without degrading the order of the optimization convergence, which helps provide a polynomial computational complexity. %

{
\begin{lemma}[Oracle complexity under approximate centers]\label{lem:buffered-implementation} \linktoproof{lem:buffered-implementation}
    Let ${\p\le\q}$ and $0<\rho\leq 1$. Within every phase of \cref{alg:lipschitz-bundle}, with fixed gap $\Gamma$ and queried bodies $\K_0\supseteq\cdots\supseteq \K_{\Nphase}$, suppose that \eqref{eq:center-movement} holds for reference centers $\z_t\in \K_t$ and that the selector provides $\newtarget{def:approximate-center}{\zhat_t}\in \K_t$ satisfying
    \begin{equation} \label{eq:buffered-center-hypothesis-approx}
        \norm{\zhat_t-\z_t}_{\q}\leq \frac{\Gamma}{16\G},\qquad 0\leq t\leq \Nphase.
    \end{equation}
  Then \cref{alg:lipschitz-bundle} with selector sequence $(\zhat_t)_t$ has oracle complexity \eqref{eq:oracle-compl-from-movement}.
\end{lemma}
}

In the sequel, we show our movement bounds and efficient implementation in order to obtain an optimization algorithm via \cref{alg:lipschitz-bundle}.

\section{Warm-up: the Steiner point in \texorpdfstring{$\ball{1}{\d}$}{the l1 ball}}
\label{sec:steiner-warmup}

For a compact convex set $\K$, its support function is defined as 
$
  \newtarget{def:support-function}{\supporth_{\K}}(\theta):=\sup_{x\in \K}\ip{\theta}{x},
$
for
$
  \theta\in\Sphere^{\d-1}.
$
Here $\newtarget{def:sphere}{\Sphere^{\d-1}}=\{\theta:\norm{\theta}_2=1\}$
and $\newtarget{def:sphere-measure}{\sphmu}$ is its uniform probability measure.  The Steiner point is
\begin{equation}\label{eq:steiner-definition}
  \newtarget{def:steiner}{\st}(\K)
  =\d\int_{\Sphere^{\d-1}}\supporth_{\K}(\theta)\theta\,\mathrm{d}\sphmu(\theta)
  =\E_{\Z\sim\cN(0,\Id_{\d})}
    \Big[\argmax_{x\in \K}\ip{\Z}{x}\Big] \in \K,
\end{equation}
where feasibility holds because this is an average of points in $\K$.  The identity follows from the divergence theorem, see \citep{przeslawskiYost1989}. We also define the \emph{mean-width} of a  compact set $\K$ as the average length of its one-dimensional projections:
\begin{equation*}
  \newtarget{def:mean-width}{\omega}(\K)
  =\int_{\Sphere^{\d-1}}
    \bigl(\supporth_{\K}(\theta)+\supporth_{\K}(-\theta)\bigr)\,\mathrm{d}\sphmu(\theta).
\end{equation*}
The following is an immediate consequence of the result in \citep{bubeck2020nested}.

\begin{fact}[Steiner Euclidean movement]
\label{thm:steiner-b1}
There is $C > 0$ such that for every $\T\ge1$ and every nested sequence of convex bodies 
\(
  \ball{2}{\d}\supseteq \K_0\supseteq \K_1\supseteq\cdots\supseteq \K_{\T}\ne\varnothing,
\)
\begin{equation*}
  \sum_{t=0}^{\T-1}\norm{\st(\K_{t+1})-\st(\K_t)}_2
  \le
  C \omega(\K_0) \sqrt{\d\log\!\left(\frac{2e\T}{\omega(\K_0)}\right)}.
\end{equation*}
\end{fact}

For an $\ell_{\p}$ ball, duality gives
\(
  \omega(\ball{\p}{\d})=2\E_{\theta\sim\sphmu}\norm{\theta}_{\p^{\conj}}.
\)
In particular,
\(
  \omega(\ball{1}{\d})
  =2\E_{\theta\sim\sphmu}\norm{\theta}_\infty
  =\Theta(\sqrt{\log(2\d)/\d}),
\)
whereas, for every fixed $1<\p<2$,
\(
  \omega(\ball{\p}{\d})=\Theta_{\p}(\d^{1/\p^{\conj}-1/2}).
\)
Since $\omega(\K_0)\le\omega(\ball{\p}{\d})$ for $\K_0\subseteq \ball{\p}{\d}$, the width-sensitive term gives $\bigOtilde(1)$ movement for $\p=1$, which in combination with \cref{thm:movement-reduction} gives a rate that is optimal up to logarithmic factors.
For fixed $\p$ not trivially close to $1$, however, it
gives a polynomial-in-dimension dependence, which is too large and suboptimal when $\T\le \d$, where the optimal movement we show in this work is
$\Thetatilde_{\p}(\T^{1-1/\p})$, see \cref{thm:movement-main,prop:nondual-movement-lower}
for Euclidean movement.

The expectation identity in \cref{eq:steiner-definition} also allows for a Monte Carlo implementation.  With $\newtarget{def:sample-count}{\Nsample}$ samples per body,
draw independent $\Z_j$, solve the linear optimization problems
$ v_j\in\argmax_{x\in \K}\ip{\Z_j}{x},$ and query the average
\(  \newtarget{def:empirical-steiner}{\shat_{\Nsample}}(\K):=\frac1\Nsample\sum_{j=1}^{\Nsample}v_j.\) 
Every query is in $\K$.  Hilbert-space
concentration (the $\ell_2$ case of the argument in
\cref{prop:monte-carlo}) gives the following error bound simultaneously along a planned path, with probability
at least $1-\delta$:
\begin{equation}\label{eq:error_approximate_steiner_point}
  \norm{\shat_{\Nsample}(\K_t)-\st(\K_t)}_2\le \errE,
  \qquad 0\le t\le \T.
\end{equation}
This holds for $0<\newtarget{def:sampling-error}{\errE}\le1$, if we use $\Nsample=\bigO(\errE[-2]\log((\T+1)/\delta))$ samples per body when
$\K_t\subseteq \ball{1}{\d}$.

For the application to our \cref{alg:lipschitz-bundle} with $\X=\ball{1}{\d}$, $\q=2$, and
$0<\eps<\G$, each body is the unit $\ell_1$-ball intersected with stored
affine halfspaces, so each sampled support point is obtained by a
polynomial-size linear program.  The uniform sampling tolerance $\errE=\eps/(32\G)$ satisfies the accuracy needed to keep the optimization rate due to \cref{lem:buffered-implementation}, since every active phase has gap
$\Gamma>\eps$. We can even just work with approximate solutions to the linear program that are guaranteed to be close to the actual solution so that we still have \cref{eq:error_approximate_steiner_point}.
Combining \cref{thm:steiner-b1,lem:buffered-implementation} gives 
$\T=\bigOtilde(\G/\eps)$ first-order queries to optimize a convex Euclidean-Lipschitz objective in $\ball{1}{\d}$.
With fresh samples at each query and total failure probability $\delta$,
the sample bound above requires $\Nsample=\bigOtilde((\G/\eps)^2)$
support-point solves per center, hence
$\bigOtilde((\G/\eps)^3)$ sampled linear programs in total. Obtaining the minimum of the model in Line~\ref{line:buffered-certified-near-minimum-model} adds only $\bigO(\T)$ linear programs.
Here the omitted factors in the $\bigotilde{\cdot}$ notation are logarithmic in $\d$, $\G/\eps$, and
$1/\delta$.

\section{Nearly optimal stable selectors for $\ell_\p/\ell_\q$-settings}\label{sec:general-center}

We construct selectors with nearly optimal high-dimensional movement bounds
for every $\ell_{\p}/\ell_{\q}$ geometry.  The mixed-norm Gaussian selector
in \cref{eq:gaussian-center} handles $\p\leq \min\{\q,2\}$, where the
Steiner-point construction does not generally attain the desired bounds.
Potential-energy centers handle $\p\ge\min\{\q,2\}$.  The resulting
movement guarantees yield nearly optimal optimization rates through
\cref{thm:movement-reduction} in the regimes where that reduction applies.

To begin, we prove a structural threshold decomposition of the $\ell_\p$-ball into
large coordinates bounded in $\ell_1$-norm (the  head) while the remaining
coordinates are small in the movement norm (the tail). The result below is elementary, and it corresponds to Stechkin's Lemma from approximation theory \citep{Stechkin:1955,Temlyakov:1986}. See also \citep[Proposition 2.3]{Foucart:2013}.

\begin{lemma}[Head-tail decomposition]
\label{lem:head-tail}\linktoproof{lem:head-tail} Let $\T\ge1$, $1\le \p\le \newtarget{def:head-tail-exponent}{\stail}\le\infty$, $\newtarget{def:threshold}{\tau}=\T^{-1/\p}$, and $x\in \RR \ball{\p}{\d}$.  Let
\(
  u_i=x_i\1_{\{\abs{x_i}>\RR\tau\}},
\) and 
\(
v=x-u.
\)
Then 
\begin{equation*}
  \norm{u}_1\le \newtarget{def:head-radius}{\Rhead}\defi \RR\T^{1-1/\p},
  \qquad
    \norm{v}_{\stail}\le \newtarget{def:tail-radius}{\Rtail_{\stail}}\defi \RR\T^{1/\stail-1/\p} \quad \text{and}\quad  \RR \ball{\p}{\d}\subseteq \Rhead \ball{1}{\d}+\Rtail_{\stail}\ball{\stail}{\d}.
\end{equation*}
\end{lemma}
Now assume ${1\le\p\le\min\{\q,2\},\q \in (1, \infty]}$ and set
\begin{equation}\label{eq:unified-exponents}
  \newtarget{def:tail-exponent}{\rmin}=\min\{2,\q\},
  \qquad
  \newtarget{def:log-dimension}{\Ld}=\left\lceil\log(2\d+2)\right\rceil,
  \qquad
  \newtarget{def:head-exponent}{\snear}=1+\frac1{\Ld}.
\end{equation}
The ceiling and the shift inside the logarithm only avoid the degeneracy at
$\d=1$ and make $\snear$ rational.  Thus
$1<\snear\le2$ and ${\p\le\rmin\le2, \rmin>1}$.
For a nonempty compact convex set $\K\subseteq \RR \ball{\p}{\d}$, we use the lift
\begin{equation}\label{eq:lift-constraints}
  \newtarget{def:lift}{\Lift_{\p,\q,\T}}(\K)
  =
  \set{(u,v):
    u+v\in \K,\ \norm{u}_1\le \Rhead,\ \norm{v}_{\rmin}\le \Rtail_{\rmin}}.
\end{equation}

Define
\begin{equation}\label{eq:gaussian-parameters}
  \newtarget{def:mu}{\mu}=\frac{\Rtail_{\rmin}^2}{\Rhead[2]}=\T^{2/\rmin-2},
  \qquad
  \newtarget{def:eta}{\eta}=\frac{\Rtail_{\rmin}^2}{\Rhead\sqrt{\Ld}}
  =\frac{\RR}{\sqrt{\Ld}} \T^{2/\rmin-1/\p-1},
\end{equation}
and, for $\newtarget{def:perturbation}{\gdir}\in\R^{\d}$, let
\begin{equation}\label{eq:unified-sample-program}
    \newtarget{def:sample-optimizer}{(\usample_{\K}(\gdir),\vsample_{\K}(\gdir))} = \argmin_{(u,v) \in \Lift_{\p,\q,\T}(\K)} \setl{ \newtarget{def:sample-objective}{\F_{\gdir}}(u,v)
  :=
  \frac12\norm{v}_{\rmin}^2
  +\frac\mu2\norm{u}_{\snear}^2
    -\eta\ip{\gdir}{u}}
\end{equation}
The lift is nonempty by
\cref{lem:head-tail}, and the objective $\F_{\gdir}$ is strongly convex %
by \cref{fact:unif-cvx-ell-p}.  Define the following deterministic selector, that we call the mixed-norm Gaussian center, as
\begin{equation}\label{eq:gaussian-center}
  \newtarget{def:gaussian-center}{\Cgauss}(\K)
  =
  \E_{\Z\sim\cN(0,\Id_{\d})}[\usample_{\K}(\Z)+\vsample_{\K}(\Z)].
\end{equation}
Every sampled point lies in $\K$, so the center lies in $\K$ as well. In the following, we show this selector is stable and later we show that we can approximate it with high probability without sacrificing the order of this movement.

\begin{theorem}[Mixed-norm Gaussian movement for ${\p\le2}$]
\label{prop:low-p-center}\linktoproof{prop:low-p-center}
    Let ${1\le\p\le\min\{\q,2\},\q\in(1,\infty]}$, and consider \cref{eq:gaussian-center}.  Every nested sequence in
$\RR \ball{\p}{\d}$ satisfies
\begin{equation*}
  \sum_{t=0}^{\T-1}
  \norm{\Cgauss(\K_{t+1})-\Cgauss(\K_t)}_{\q}
  =
  \bigotildelp{\p,\q}{
    \RR \T^{1-\frac1\p+\left(\frac1\q-\frac12\right)_{\pospart}}
  }.
\end{equation*}
\end{theorem}

In the analysis, we control the movement of the head in $\ell_1$ using the strong convexity of $\norm{u}_{\snear}^2/2$ with respect to $\norm{\cdot}_1$, cf. \cref{eq:unified-sample-program}, and Gaussian integration by parts controls also the head movement but in $\ell_2$. An interpolation of both bounds along with controlling the movement of the tail in the $\min\{\q, 2\}$-norm, yields the result.

\subsection{The potential-energy center for $\p \geq \min\{\q, 2\}$}

For $\p\ge\min\{\q,2\}$, convex optimization is already understood up
to logarithmic factors through mirror descent.  Minimizing a uniformly
convex regularizer over the current request is a natural selector:
mirror descent and follow-the-regularized-leader (FTRL) algorithms use
related regularization potentials to control optimization iterates.
Here the same convexity estimates control movement along nested requests,
giving nearly optimal movement bounds in these geometries.

Let $\Ld=\lceil\log(2\d+2)\rceil$. For $1\le\p,\q\le\infty$
with $\p\ge\min\{\q,2\}$, take the value
$\min\{\p,\max\{\q,2\}\}$ and clip it away from $1$ or $\infty$, defining  $a=\max\{1+\frac{1}{\Ld},\min\{\Ld,\p,\max\{\q,2\}\} $.
Let $K$ be non empty, compact and convex. We use the following differentiable convex regularizer and selector, respectively:  
\begin{equation}\label{eq:energy-complement-selector}
  \newtarget{def:energy-regularizer}{\Psipq}(x):=
  \begin{cases}
    \frac12\norm{x}_a^2, & 1<a\le2,\\[1mm]
    \frac1a\norm{x}_a^a, & 2<a<\infty,
  \end{cases} \qquad \qquad
  \newtarget{def:adapted-energy-center}{\Ppq}(\K):=\argmin_{x\in\K}\Psipq(x).
\end{equation}
Compactness and the uniform convexity properties in
\cref{fact:unif-cvx-ell-p} yield a unique minimizer above.

\begin{proposition}[Movement of adapted energy centers]
\label{prop:energy-complement}\linktoproof{prop:energy-complement}
Let $1\le\p,\q\le\infty$ satisfy $\p\ge\min\{\q,2\}$, and let $\RR>0$.  For nonempty compact convex nested sets
$\RR\ball{\p}{\d}\supseteq\K_0\supseteq\cdots\supseteq\K_{\T}$,
let $\z_t=\Ppq(\K_t)$.  Then
\begin{equation}\label{eq:energy-complement-movement}
  \sum_{t=0}^{\T-1}\norm{\z_{t+1}-\z_t}_{\q}
  =\bigotildep{\p}{\RR}
  \begin{cases}
    \T^{1-1/\p}, & 2\le\p<\q,\\[1mm]
    \d^{1/\q-1/\p}\T^{1-1/\max\{\q,2\}}, & 1\le\q\le\p\le\infty.
  \end{cases}
\end{equation}
\end{proposition}

The movement above matches the lower bounds in \cref{prop:nondual-movement-lower} when $\T\le\d$, up to constants depending on $\p$ for $\q>1$, and up to an additional logarithmic factor $\sqrt{\Ld}$ for $\q=1$.%
Together with the mixed-norm Gaussian bound in \cref{prop:low-p-center}, this gives nearly optimal movement for all $\ell_\p / \ell_\q$ cases.

\section{Feasible randomized implementation}
\label{sec:randomized-implementation}

The results that we have shown so far only provide the existence of a strategy with bounded movement, with the exception of the Steiner point example and the potential-energy selector. Here we focus on a simple Monte Carlo approximation of the mixed-norm Gaussian center.

We recall the notation and context from the previous section.
Let ${1\le\p\le\min\{\q,2\}, \q\in(1,\infty]}$, $\rmin=\min\{\q,2\}$, $\T\ge1$,
and $0<\delta<1$.  Let
$\K_0\supseteq\cdots\supseteq\K_{\T}$ be nonempty compact convex subsets of
$\RR\ball{\p}{\d}$, possibly chosen adaptively.  Let $0<\errE\le\RR$
be the sampling tolerance: it controls the error in approximating the
center's expectation by an average of points solving \cref{eq:unified-sample-program} exactly.
Choose the number of samples per body so that
\begin{equation}\label{eq:monte-carlo-sample-size}
  \Nsample\ge C_{\rmin}\left[
    (\RR/\errE)^{\rmin^{\conj}}
    +(\RR/\errE)^2\log\!\left(\frac{\T+1}{\delta}\right)
  \right],
\end{equation}
where $C_{\rmin}>0$ depends only on $\rmin$.  After $\K_t$ is known, draw
$\Nsample$ fresh independent standard Gaussian vectors
$\Z_{t,1},\ldots,\Z_{t,\Nsample}$ and let
\[
  \newtarget{def:center-sample}{\Y_{t,j}}
  =\usample_{\K_t}(\Z_{t,j})+\vsample_{\K_t}(\Z_{t,j}),
  \qquad 1\le j\le\Nsample,
\]
using the definitions in \cref{eq:unified-sample-program}.
Compute feasible numerical approximations $\Yhat_{t,j}$ of $Y_{t,j}$ up to error $\newtarget{def:solver-error}{\Esol}\ge0$ and define the randomized selector $\zhat_t$ by
\[
  \begin{aligned}
      \zhat_t=\frac1\Nsample\sum_{j=1}^{\Nsample}\Yhat_{t,j} \in \K_t, \ \ 0\le t\le\T, \qquad \text{ where } \ \
    \newtarget{def:numerical-sample}{\Yhat_{t,j}}\in\K_t, \ \ 
    \norm{\Yhat_{t,j}-\Y_{t,j}}_{\q}\le\Esol,
    \ \ 1\le j\le\Nsample.
  \end{aligned}
\]
On the sampling success event, the two errors add, giving center error
at most $\errE+\Esol$.
\begin{proposition}[Movement of the feasible randomized selector]
\label{prop:monte-carlo}\linktoproof{prop:monte-carlo}
    Let $\delta \in (0, 1)$,
\(\errE=\Esol
    =\frac{1}{4}\RR \T^{-\frac1\p+(\frac1\q-\frac12)_{\pospart}}\), and choose $N$ samples as in \cref{eq:monte-carlo-sample-size}.
Then, with probability at least $1-\delta$,
\begin{equation}\label{eq:monte-carlo-movement}
  \sum_{t=0}^{\T-1}\norm{\zhat_{t+1}-\zhat_t}_{\q}
  =
  \bigOtilde_{\p,\q}\!\left(
    \RR\T^{1-\frac1\p+(\frac1\q-\frac12)_{\pospart}}
  \right).
\end{equation}
\end{proposition}
In a bundle phase with gap $\Gamma$ from \cref{eq:bundle-levels}, we want to apply \cref{lem:buffered-implementation} so we impose the sampling and solver errors to satisfy:
\begin{equation}\label{eq:bundle-center-tolerance}
  \errE+\Esol\le\Gamma/(16\G).
\end{equation}
Since every active phase has $\Gamma>\eps$, the tolerances required and the sample count in \cref{eq:monte-carlo-sample-size} have polynomial dependence on inverse target accuracy for fixed $\p,\q$. Conditionally on the past, the body is fixed, so a union bound therefore remains
valid against an adaptive oracle.

For the next theorem, we assume access to a first-order oracle for $\f$ and a real-arithmetic
model in which arithmetic, comparisons, rational powers, logarithms, and
independent scalar standard Gaussian draws have unit cost.

\begin{theorem}[Polynomial-time implementation]
\label{thm:polynomial-implementation}\linktoproof{thm:polynomial-implementation}
    Fix ${1\le\p\le\min\{\q,2\}}$, $\q \in (1, \infty]$ and $\eps>0$, $\delta\in(0,1)$. Let $\f:\R^{\d} \to \R$ be convex and $\G$-Lipschitz in $\norm{\cdot}_{\q}$. There is a randomized implementation of \cref{alg:lipschitz-bundle} using
    $\bigOtilde_{\p,\q}\!\left(1+(\G\RR/\eps)^{1/(1/\p-(1/\q-1/2)_{\pospart})}\right)$ first-order oracle calls and
$\operatorname{poly}_{\p,\q}(\d,1+\G\RR/\eps,\log(1/\delta))$ additional real operations for minimization over $\RR\ball{\p}{\d}$. All queries and the output $\xbest$ belong to $\RR\ball{\p}{\d}$.
With probability at least $1-\delta$, it returns a certificate
\[
    \lb\le \min_{x \in \RR \ball{\p}{\d}}\f(x)\le\f(\xbest)\le\lb+\eps.
\]
\end{theorem}

The oracle rates above do not depend on $\log(1 / \delta)$. The polynomial implementation comes from approximating the solution of the convex program by the ellipsoid algorithm, and the polynomial degree and constants depend only on $\p,\q$. We did not attempt to optimize the polynomial in the implementation and just used the ellipsoid method for simplicity in order to provide a polynomial-time algorithm. The complexity of our ellipsoid method depends critically on inner and outer balls for the feasible domain. Our proof certifies quantifiable bounds on the queried bodies.

\paragraph{Subproblem implementation.}
The body $\K$ is an $\ell_{\p}$ ball intersected with stored affine halfspaces, so
it has an explicit separation oracle.  The lower bound
$\min_\X\model$ is a convex epigraph problem.  The Monte Carlo construction
preceding \cref{prop:monte-carlo} uses the phase tolerance in
\cref{eq:bundle-center-tolerance} and remains feasible by convexity.
\Cref{alg:lipschitz-bundle} uses certified model-minimum bounds and the stopping test in line~\ref{line:buffered-bundle-emptiness}.  \Cref{thm:movement-reduction} analyzes selectors with a movement
bound, and \cref{lem:buffered-implementation} allows approximate centers.
For ${\p\le\min\{\q,2\}, \q\in(1,\infty]}$, the convex subproblems can be solved in polynomial time
by \cref{thm:polynomial-implementation}.

\section{Near optimality of our randomized algorithms in the local oracle model}
\label{sec:lower-bounds}

Note that the oracle lower bounds of the open problem in \citep{guzman2015thesis,guzman2015open} apply exclusively to deterministic algorithms, such as the one for our deterministic selector in \cref{eq:gaussian-center} with the full expectation.
Following the approach in \citep{Braun:2017}, we extend those lower bounds to work against randomized algorithms, to show the optimality of our efficient randomized implementations. We also leverage a general lower bound strategy pioneered by \citet{srebro2012convex}. For the result we only discuss the high-level ideas, and refer to the original references for further details.

We briefly introduce the setting of oracle complexity \citep{nemirovskiYudin1983}. 
Let $\newtarget{def:function-sets}{\cX,\cG}\subseteq \R^\d$ be convex and let $\cG$ satisfy $\cG=-\cG$, and consider the class of convex Lipschitz functions 
\[ \newtarget{def:function-class}{\cF}(\cX,\cG)\defi \{f:\R^\d\to\R:\, f\mbox{ is convex, and } \forall x\in\cX, \partial f(x)\subseteq\cG\}.\]
Let $\newtarget{def:local-oracles}{\Oracle}$ be the class of local oracles for the class above. Here, locality refers to the property that the oracle cannot distinguish functions that coincide locally around the query point. Formally, an oracle $\oracle$ is local if $\oracle_f(x)=\oracle_g(x)$ whenever $f,g\in\cF(\cX,\cG)$ agree on a ball $B(x,\delta)$ with $\delta>0$. Let $\newtarget{def:randomized-algorithms}{\RandAlg}$ be the class of {\em randomized algorithms}, which we regard as mixtures of deterministic algorithms. This is equivalent to the use of random bits in computation, by exhausting the computational search tree.
For $\T\geq1$, let $\RandAlg_{\T}(\oracle)$ denote
the class of randomized algorithms that make at most $\T$ adaptive
queries to $\oracle$ at points in $\cX$, for every instance
and every realization of their internal randomness, and then
return a point $\widehat x_{\T}(A,f)\in\cX$.
We define the expected minimax risk after $\T$ queries by
\begin{equation}\label{eq:oracle-complexity}
    \operatorname{Risk}_{\T}
    \bigl(\oracle,\cF(\cX,\cG)\bigr)
    \defi
    \inf_{A\in\RandAlg_{\T}(\oracle)}
    \sup_{f\in\cF(\cX,\cG)}
    \mathbb E_A\left[
        f\bigl(\widehat x_{\T}(A,f)\bigr)
        -\inf_{x\in\cX}f(x)
    \right].
\end{equation}

\begin{theorem}
\label{thm:LB-randomized-budget}\label{thm:LB_randomized_pq}
\linktoproof{thm:LB-randomized-budget}
Let $1\leq\p,\q\leq\infty$. There exists a constant
$c_{\p,\q}\in(0,1/2)$ such that, for every integer $\d\geq1$,
$\G,\RR>0$, every integer $1\leq\T\leq c_{\p,\q}\d$, and every
local oracle $\oracle$ for
$\cF(\RR\B_{\p}^{\d},\G\B_{\q^{\conj}}^{\d})$,
\[
    \operatorname{Risk}_{\T}\bigl(
        \oracle,
        \cF(\RR\B_{\p}^{\d},\G\B_{\q^{\conj}}^{\d})
    \bigr)
    \geq c_{\p,\q}\G\RR
    \begin{cases}
        \T^{-1/\p+(1/\q-1/2)_{\pospart}},
            & \p<\q,\\[2mm]
        \d^{1/\q-1/\p}\T^{-1/\max\{\q,2\}},
            & \q\leq\p<\infty,\\[2mm]
        1, & \p=\infty.
    \end{cases}
\]
\end{theorem}

\cref{thm:optimization-main}, applying for $\p \leq \min\{\q, 2\}$, matches this bound up to logarithmic factors. %

\section{Discussion}

This work provides the first polynomial-time guarantee with nearly optimal oracle complexity for Lipschitz convex optimization in the $\ell_{\p}/\ell_{\q}$ setting when $\p<\min\{\q,2\}$,  answering positively the nonsmooth end of a COLT 2015 open problem \citep{guzman2015open}, and realizing the rates in \citep{bullins2026firstorder} efficiently. 
Attaining these rates requires genuinely new techniques, particularly departing from the classical cumulative regret minimization approaches that seem to be unable to leverage the additional structure of the feasible domain in nondual optimization settings.

We believe our reduction of convex optimization to nested convex-set chasing to be of independent interest and we hope for this newly found connection to convex optimization to be fruitful. 
We also initiated the study of movement bounds in the high-dimensional setting where the number of steps $\T$ is no larger than the dimension $\d$, providing nearly matching upper and lower bounds in the $\ell_\p/\ell_\q$ case, a setting that is scarcely the focus of study in online algorithms.
To fully resolve the open question in \citep{guzman2015open} the oracle complexity of the entire spectrum of H\"older-smooth settings should be considered. We and coauthors addressed this extension in \citep{martinezrubio2026nearoptimalacceleration} using the results of the present work.

\clearpage
\section*{Acknowledgements}
David Martínez-Rubio was funded by a fellowship of the Bank Foundation ``la Caixa'' (ID 100010434) with code LCF/BQ/PI25/12100033. He thanks OpenAI for free access to their models. Crist\'obal Guzm\'an was partially funded by ANID FONDECYT 1251029 grant, and ANID Basal FB210017 National Center for Artificial Intelligence CENIA.
\AI{%
This work was elaborated in combination with ChatGPT/Codex versions 5.5 and 5.6 (the newest models when we worked on the main results), that helped in a few places after a highly interactive workflow. Most of the main ideas were from the authors (adapting the bundle method from \citep{lemarechal1995new}, connection to chasing nested convex bodies along with deep cuts, studying high-dimensional movement bounds, head and tail decomposition,...), while AI was used to accelerate computations, to quickly close fruitless directions, doing checks and elaboration of some proofs. In particular, an initial version of the mixed-norm Gaussian movement analysis was fully done by AI. The authors checked and rewrote the proofs that were automated, simplifying the presentation and making the main ideas more transparent. In particular, they are solely responsible for the contents of this manuscript. 

}

\appendix
\ifusecolttemplate\crefalias{section}{appendix}\fi

\begingroup
\renewcommand{\mu}{\newlink{def:uniform-convexity-modulus}{\oldmu}}
\section{Facts About Strongly/Uniformly Convex Functions}

Here we summarize some classical facts about strongly convex and uniformly convex functions used in the proofs.

\begin{definition}
Let $(E,\|\cdot\|)$ be a normed space, $\K\subseteq E$ a closed convex set, and $\newtarget{def:uniform-convexity-modulus}{\mu}\geq 0$, $2\leq r<\infty$. We say that $\Phi:\K\to\R$ is $(\mu,r)$-uniformly convex if for all $0\leq \lambda\leq 1$ and $x,y\in\K$:
\[ \Phi\big((1-\lambda)x+\lambda y\big) \leq (1-\lambda)\Phi(x)+\lambda \Phi(y) -\frac{\mu}{r}\lambda(1-\lambda)\bigl(\lambda^{r-1}+(1-\lambda)^{r-1}\bigr)\|y-x\|^r. \]
We say that a function is $\mu$-strongly convex if it is $(\mu,2)$-uniformly convex.
\end{definition}

A useful alternative characterization is as follows. The next result is a particular case of \citep[Corollary 3.5.11]{Zalinescu:2002}.
\begin{proposition}  \label{prop:Bregman-LB-unif-cvx}
    Let $\Phi:E\to\R$ be differentiable. Then $\Phi$ is $(\mu,r)$-uniformly convex if and only if for all $y,x\in E$
    \[ \D_{\Phi}(y,x)\geq \frac{\mu}{r}\|y-x\|^r,\]
    where $\newtarget{def:bregman}{\D_{\Phi}}(y,x)\defi\Phi(y)-\Phi(x)-\ip{\nabla \Phi(x)}{y-x}$  
    is the Bregman divergence.
\end{proposition}

We summarize the uniform convexity properties of $\ell_{\p}^{\d}=(\R^{\d},\norm{\cdot}_{\p})$ spaces, for $1\leq \p<\infty$. These results are classical in Analysis \citep{clarkson1936uniformly,ballCarlenLieb1994}.
\begin{fact} \label{fact:unif-cvx-ell-p}
  \begin{enumerate}
      \item If $1<\p\leq2$, then the function $\newtarget{def:p-energy}{\Phip}(x)\defi \frac12\norm{x}_{\p}^2$ is $(\p-1)$-strongly convex (see e.g.~\citep [Example 5.28]{Beck:2017}).
      \item If $2<\p<+\infty$, then the function $\Phip(x)=\frac1\p\norm{x}_{\p}^{\p}$ is $(2^{2-\p},\p)$-uniformly convex (see e.g., \citep[Lemma 4.2.3]{Nesterov:2018lectures}).
  \end{enumerate}  
\end{fact}

While the extreme cases $\p=1,\infty$ are not included above, the case $\p=1$ is known to have strongly convex functions with moderate growth. We provide the details below for completeness.

\begin{fact}[Near-strong convexity of $\ell_1$-norm]
\label{lem:near-l1-geometry}
Let $\snear$ and $\Ld$ be as in \cref{eq:unified-exponents}.  Then
$x\mapsto \frac12\norm{x}_{\snear}^2$ is $e^{-2}/\Ld$-strongly convex with respect
to $\norm{\cdot}_1$; more precisely,
\begin{equation*}
  \D_{\frac12\norm{\cdot}_{\snear}^2}(u',u)
  \ge
  \frac{\snear-1}{2}\norm{u'-u}_{\snear}^2
  \ge
  \frac{e^{-2}}{2\Ld}\norm{u'-u}_1^2.
\end{equation*}
Moreover, $0\le\frac12\norm{u}_{\snear}^2\le \Rhead[2]/2$ whenever
$\norm{u}_1\le \Rhead$.
\end{fact}

\begin{proof}
By \cref{fact:unif-cvx-ell-p}, for any $1<\snear\le2$, the squared $\ell_{\snear}$ norm is $(\snear-1)$-strongly convex with
respect to $\norm{\cdot}_{\snear}$.  Since $\snear-1=1/\Ld$ and
\begin{equation*}
  \d^{2/\snear-2}
  =
  \exp\left(-\frac{2(\snear-1)}\snear\log \d\right)
  \ge e^{-2},
\end{equation*}
we have $(\snear-1)\d^{2/\snear-2}\ge e^{-2}/\Ld$.  Combining this with the norm
comparison $\norm{x}_{\snear}\ge \d^{1/\snear-1}\norm{x}_1$ proves the second
inequality.  The range bound follows from $\norm{u}_{\snear}\le\norm{u}_1\le \Rhead$.
\end{proof}

\endgroup

\section{Fat-Shattering and Sequential Fat-Shattering Dimensions}\label{sec:fat_shattering_and_lower_bounds}

Some results we provide work in much broader generality than $\ell_\p/\ell_\q$-settings. These generalizations are useful and insightful, thus we choose to present them at this level of generality, providing the necessary background here.
Let $(E,\langle\cdot,\cdot\rangle)$ be a finite-dimensional space. For $\cX,\cG\subseteq E$ convex and centrally symmetric, consider the class of evaluation functionals
\[ \newtarget{def:linear-class}{\cL}(\cX, \cG) \defi\{g\mapsto \langle g,x\rangle ,\,\, x\in \cX\}, \]
as well as the class of convex Lipschitz functions 
\[ \cF(\cX,\cG)\defi \{f:E\to\R:\, f\mbox{ is convex and } (\forall x\in\cX)\, \partial f(x)\subseteq\cG\}.\]

We also recall the notion of fat-shattering dimension of linear functionals, which corresponds to a scale-sensitive counterpart of the VC-dimension used in statistical learning.
\begin{definition} \label{def:fat-shattering}
    A set $g_1,\ldots,g_K\in \cG$ is $\alpha$-shattered by $\cL(\cX,\cG)$ if there exist $r_1,\ldots,r_K\in \R$ such that for every $s\in\{+1,-1\}^K$ there exists $x_s\in \cX$ such that for all $i\in [K]$,
    \[s_i(\langle g_i,x_s\rangle-r_i)>\alpha/2.\]
    We define $\newtarget{def:fat-shattering}{\fat}_{\alpha}(\cL(\cX,\cG))$ as the largest value of $K$ such that there is a set $g_1,\ldots,g_K\in \cG$ that is $\alpha$-shattered by $\cL(\cX,\cG)$.
\end{definition}
The following notion is a scale-sensitive counterpart to the Littlestone dimension used in online learning.
\begin{definition}[Sequential fat-shattering]
\label{def:sequential-fat-shattering}
Fix $\alpha>0$. A $\cG$-valued binary tree of depth $\T$ is a family
$\mathbf g=(\mathbf g_t)_{t=1}^{\T}$ of maps
\[
  \mathbf g_t:\{-1,1\}^{t-1}\to\cG.
\]
For $\sigma\in\{-1,1\}^{\T}$, denote
$\sigma_{<t}=(\sigma_1,\ldots,\sigma_{t-1})$, with
$\sigma_{<1}$ the empty sequence.
The tree $\mathbf g$ is $\alpha$-shattered by $\cL(\cX,\cG)$ if there
exist maps
$\mathbf r_t:\{-1,1\}^{t-1}\to\R$ such that, for every
$\sigma\in\{-1,1\}^{\T}$, there exists 
$x_{\sigma}\in\cX$ satisfying
\[
  \sigma_t\bigl(
    \ip{\mathbf g_t(\sigma_{<t})}{x_{\sigma}}
    -\mathbf r_t(\sigma_{<t})
  \bigr)>\frac{\alpha}{2}
  \qquad\text{for all }t\in[\T].
\]
The sequential fat-shattering dimension
$\newtarget{def:sequential-fat-shattering}{\sfat{\alpha}}(\cL(\cX,\cG))$ is the largest such
depth $\T$, with value $0$ if no tree of depth $1$ is shattered and
$\infty$ if arbitrarily deep trees are shattered.
\end{definition}

Taking both node maps to be constant across each level recovers ordinary
fat-shattering. In particular,
\[
  \fat_{\alpha}(\cL(\cX,\cG))
  \le \sfat{\alpha}(\cL(\cX,\cG)).
\]
The sequential version allows each direction and threshold to depend
on the previously chosen signs.

\paragraph{Fat-Shattering and sequential fat-shattering dimensions of $\ell_\p/\ell_\q$ spaces.} Let  $1\le \p,\q \le \infty$. For the case of $\cX=\ball{\p}{\d}$ and $\cG=\ball{\q^{\conj}}{\d}$, we provide bounds on the dimension of the class of evaluation functionals. In what follows, let $\log_{\pospart}t\defi (\log t)_{\pospart}$, and $L_{\alpha}\defi 1+\log_{\pospart}(e\d/\alpha)$.

Let
\(
  \Lambda=\sup_{x\in\ball{\p}{\d},\,g\in\ball{\q^{\conj}}{\d}}
    |\ip{g}{x}|=\d^{(1/\q-1/\p)_{\pospart}}.
\)
The two tables below hold for $0<\alpha\le c_{\p,\q}\Lambda$, where
$c_{\p,\q}\in(0,1/2]$ is a sufficiently small constant depending only
on $\p,\q$.  Both dimensions are zero for $\alpha\ge2\Lambda$,
since the shattering inequalities require strict separation.

For fat-shattering dimension, results from \citep{mendelson2004shattering,guzman2015thesis} imply
\begin{equation}\label{eq:fat_shattering_values}
\fat_\alpha(\cL(\ball{\p}{\d},\ball{\q^{\conj}}{\d}))
\asymp_{\p,\q}
\begin{cases}
\min\left\{
    \alpha^{-\frac{1}{\frac{1}{\p}-(\frac{1}{\q}-\frac{1}{2})_{\pospart}}},\d
\right\},
& \p<\q,\\[1mm]
\min\!\left\{
    \left(\d^{1/\q-1/\p}/\alpha\right)^{\max\{\q,2\}},\d
\right\},
& \p\ge\q,\ (\p,\q)\notin\{(1,1),(\infty,\infty)\}.
\end{cases}
\end{equation}
At $(\p,\q)=(1,1)$, the results of
\citet[Corollary~2.3 and Theorem~3.1]{mendelson2004shattering}
give the sharp characterization
\begin{equation}\label{eq:fat-shattering-11}
    \fat_\alpha(\cL(\ball{1}{\d},\ball{\infty}{\d}))
    \asymp
    \min\left\{\d,\frac{\log(2+\d\alpha^2)}{\alpha^2}\right\},
    \qquad 0<\alpha\le1.
\end{equation}

At $(\p,\q)=(\infty,\infty)$, the ordinary dimension is $\fat_\alpha(\cL(\ball{\infty}{\d},\ball{1}{\d}))= \d$ if $0<\alpha<2$ and $0$ if $\alpha\ge2$.

For sequential fat-shattering dimensions, we borrow the bounds from \citep{bullins2026firstorder}, where it is shown that $\sfat{\alpha}(\cL(\ball{\p}{\d},\ball{\q^{\conj}}{\d}))$ is bounded up to constants depending only on $\p,\q$, between the following quantity and $L_{\alpha}$ times it:
\begin{equation}\label{eq:seq_fat_shattering_values}
\begin{cases}
\min\left\{
    \alpha^{-\frac{1}{\frac{1}{\p}-(\frac{1}{\q}-\frac{1}{2})_{\pospart}}},\d\Big(1+\log_{\pospart}\frac{1}{\alpha \d^{\frac1\p-(\frac1\q-\frac12)_{\pospart}}}\Big)
\right\},
& \p<\q,\\[1mm]
\min\!\left\{
    \left(\d^{1/\q-1/\p}/\alpha\right)^{\max\{\q,2\}},\d \Big(1+\log_{\pospart}\frac{\d^{\frac{1}{\q}-\frac{1}{\p}}}{\alpha \d^{1/\max\{\q,2\}}}\Big)
\right\},
& \p\ge\q,\ (\p,\q)\notin\{(1,1),(\infty,\infty)\}.
\end{cases}
\end{equation}
At $(\p,\q)=(\infty,\infty)$, the sequential profile is instead
\[
  \sfat{\alpha}(\cL(\ball{\infty}{\d},\ball{1}{\d}))
  \asymp \d\bigl(1+\log_{\pospart}(1/\alpha)\bigr),
  \qquad 0<\alpha\le1.
\]

\section{Proof of \Cref{prop:nondual-movement-lower}}

We first prove a more general version of the result.

\begin{proposition}[General movement and competitive ratio lower bounds]
\label{prop:sequential-movement-lower}
Let $\K_0\subseteq\R^{\d}$ be compact and convex, let
$\z_0\in\K_0$, and let $\norm{\cdot}$ be a norm on $\R^{\d}$.
For $r>0$, set
\[
  \cX_r\defi\{x\in\K_0:\norm{x-\z_0}\le r\},
  \qquad
  \cG\defi\{g\in\R^{\d}:\norm{g}_*\le1\}.
\]
For every positive integer $\T$, every possibly randomized online selector
starting from $\z_0$ has worst-case movement at least
\[
  \sup_{\alpha>0}\frac{\alpha}{2}
  \min\!\left\{\sfat{\alpha}
    \bigl(\cL(\cX_r,\cG)\bigr),\T\right\},
\]
where the worst case ranges over adaptive adversaries generating nonempty
compact convex requests $\K_0\supseteq\K_1\supseteq\cdots\supseteq\K_{\T}$
with offline optimum in the same norm at most $r$.
The constructions giving this bound hold for every realization of the
selector's randomness against adversaries observing its previous iterates.
\end{proposition}

\begin{proof}
Fix $\alpha>0$ and set
\[
  m\defi\min\!\left\{\sfat{\alpha}
    \bigl(\cL(\cX_r,\cG)\bigr),\T\right\}.
\]
If $m=0$, take $\K_t=\K_0$ for all $t$; the movement bound is zero and
the offline player can remain at $\z_0$. Otherwise, fix the trees
$\mathbf g$ and $\mathbf r$ witnessing sequential fat-shattering at depth $m$.
Thus by definition, every $\sigma\in\{-1,1\}^{m}$ admits a witness
$x_{\sigma}\in\cX_r$ such that
\[
  \sigma_t\bigl(
    \ip{\mathbf g_t(\sigma_{<t})}{x_{\sigma}}
    -\mathbf r_t(\sigma_{<t})
  \bigr)>\frac{\alpha}{2}
  \qquad(\forall t\in[m]),
\]
where $\sigma_{<t}=(\sigma_1,\ldots,\sigma_{t-1})$.
At time $t\le m$, the preceding signs determine
\[
  g_t=\mathbf g_t(\sigma_{<t}),
  \qquad
  r_t=\mathbf r_t(\sigma_{<t}).
\]
After observing $\z_{t-1}$, choose
$\sigma_t\in\{-1,1\}$ such that
\[
  \sigma_t\bigl(\ip{g_t}{\z_{t-1}}-r_t\bigr)\le0,
\]
and reveal
\begin{equation}\label{eq:sequential-shattering-cut}
  \K_t=\K_{t-1}\cap
  \left\{x:\sigma_t\bigl(\ip{g_t}{x}-r_t\bigr)
    \ge\frac{\alpha}{2}\right\}.
\end{equation}
For $m<t\le\T$, simply repeat the last request: $\K_t=\K_m$.
Extending any partial path to a full sign sequence gives a witness
satisfying every cut imposed so far, so each request is nonempty; compactness,
convexity, and nesting follow from the construction.

For each $t\le m$, feasibility, the choice of $\sigma_t$, and norm duality give
\[
  \frac{\alpha}{2}
  \le\sigma_t\ip{g_t}{\z_t-\z_{t-1}}
  \le\norm{\z_t-\z_{t-1}}.
\]
Summing these $m$ transitions and using nonnegativity of movement on the
remaining rounds gives the stated movement bound. The witness of the final
depth-$m$ path belongs to every request, including the repeated requests.
The offline player can move to it on the first request and remain there,
paying at most $r$.
The adversary uses only previous iterates, so these bounds hold for every
realization of the selector's randomness. Taking the supremum over $\alpha$
in the worst-case movement bound proves the claim.
\end{proof}

\begin{remark}[An oblivious lower bound in expectation]
Under the assumptions of \cref{prop:sequential-movement-lower},
one can also choose a distribution over request sequences before
playing. Sample all signs independently and uniformly, and use the
corresponding path through the same tree with the cuts in
\eqref{eq:sequential-shattering-cut}.
For every sampled path the requests remain nonempty, and the
offline cost is at most $r$.

Conditionally on the preceding signs and the selector's random
choices determining $\z_{t-1}$, the quantities $g_t,r_t$ and
$\z_{t-1}$ are fixed, while $\sigma_t$ is uniform.
Feasibility and norm duality give, for all $t\leq m$,
\[
  \norm{\z_t-\z_{t-1}}
  \ge\sigma_t\ip{g_t}{\z_t-\z_{t-1}}
  \ge\frac{\alpha}{2}
     -\sigma_t\bigl(\ip{g_t}{\z_{t-1}}-r_t\bigr).
\]
Taking conditional expectations and summing yields
\[
  \E\left[\sum_{t=0}^{m-1}
    \norm{\z_{t+1}-\z_t}\right]
  \ge\frac{\alpha}{2}\min\!\left\{
    \sfat{\alpha}(\cL(\cX_r,\cG)),\T\right\}.
\]
Consequently, for every randomized selector, there is a fixed
request sequence with expected movement at least $\frac{\alpha}{2}
  \min\!\left\{\sfat{\alpha}
    \bigl(\cL(\cX_r,\cG)\bigr),\T\right\}$
and offline cost at most $r$.
\end{remark}

\begin{proof}\linkofproof{prop:nondual-movement-lower}
Define 
\[
  r\defi
  \begin{cases}
    \RR\T^{1/\q-1/\p}, & \p\le\q,\\
    \RR\d^{1/\q-1/\p}, & \p>\q,
  \end{cases}
\]
and let
\[
  \cX\defi
  \begin{cases}
    \{x\in\R^{\d}:\norm{x}_{\q}\le r,\;
      x_i=0\text{ for }i>\T\}, & \p\le\q,\\
    \RR\ball{\p}{\d}, & \p>\q.
  \end{cases}
\]
Norm comparison gives
$\cX\subseteq\cX_r\defi\RR\ball{\p}{\d}\cap r\ball{\q}{\d}$.

Recall that the fat-shattering dimension lower bounds the sequential fat-shattering dimension. By the sequential fat-shattering bounds in \eqref{eq:seq_fat_shattering_values} and also \eqref{eq:fat-shattering-11} for $\p=\q=1$, there is a constant $a_{\p,\q}>0$ such that, for
    $\alpha=a_{\p,\q}r\T^{-1/\max\{\q, 2\}}$,
\[
  \sfat{\alpha}
    \bigl(\cL(\cX,\ball{\q^{\conj}}{\d})\bigr)\ge\T.
\]
For $\p>\q$, this follows directly from its $(\p,\q)$
case in dimension $\d$, followed by scaling by $\RR$. 
For $\p\le\q$, this follows from its diagonal $(\q,\q)$ case in
dimension $\T$, followed by scaling by $r$ and embedding in
$\R^{\d}$. 

By monotonicity, the same shattering bound holds with $\cX_r$ in
place of $\cX$. Applying \cref{prop:sequential-movement-lower}
with movement in $\ell_{\q}$ yields
\[
  \sum_{t=0}^{\T-1}\norm{\z_{t+1}-\z_t}_{\q}
  \ge\frac{\alpha\T}{2}
    =\frac{a_{\p,\q}}{2}\,r\T^{1-1/\max\{\q, 2\}},
\]
with offline cost at most $r$. Taking $c_{\p,\q}=a_{\p,\q}/2$
and substituting the definition of $r$ gives the stated movement bound, since
\[
  1+\frac1\q-\frac1\p-\frac1{\max\{\q,2\}}
  =1-\frac1\p+\left(\frac1\q-\frac12\right)_{\pospart}.
\]
\Cref{prop:sequential-movement-lower} gives a comparator with
$\ell_{\q}$-norm at most $r$, so the competitive ratio is at least
\[
  \frac{c_{\p,\q}\,r\T^{1-1/\max\{\q,2\}}}{r}
  =c_{\p,\q}\T^{1-1/\max\{\q,2\}}.
\]
Both bounds hold for every realization of the selector's randomness,
as in \cref{prop:sequential-movement-lower}.
\end{proof}

\section{Proofs from \Cref{sec:bundle}}

\begin{proof}\linkofproof{lem:null-deep}
The halfspace is $\set{x:\f(\z)+\ip{\g}{x-\z}\le \hminus}$. If $\g=0$, then $\f(\z)>\hminus$ makes the halfspace empty, so the claim is immediate. Suppose $\g\ne0$. %
Let $x\in \set{w:\f(\z)+\ip{\g}{w-\z}\le \hminus}$. Note that
    \[\frac{\Gamma}{2}=\hplus-\hminus \circled{1}[<]\f(\z)-\hminus \circled{2}[\le] \ip{\g}{\z-x}\circled{3}[\le] \norm{\g}_{\q^{\conj}}\norm{\z-x}_{\q}\leq \G\norm{\z-x}_{\q},\]
    where the first equality uses the phase levels in \cref{eq:bundle-levels}, $\circled{1}$ uses the fact that $\f(\z)>\hplus$, $\circled{2}$ uses that $x$ lies on the halfspace defined above, and $\circled{3}$ uses Hölder's inequality. Equality in the distance formula is attained at
$x=\z-(\f(\z)-\hminus)u/\norm{\g}_{\q^{\conj}}$, where
$\norm{u}_{\q}=1$ and $\ip{\g}{u}=\norm{\g}_{\q^{\conj}}$
by finite-dimensional norm duality.
\end{proof}

\begin{proof}\linkofproof{thm:movement-reduction}
Fix a target accuracy $\eps>0$.  Consider a phase whose certified gap is
$\Gamma$, and let $\lb_0$ be the lower bound at the start of the phase.  If $\Nphase$ consecutive \nullstep{}
cuts leave the bodies nonempty, \cref{lem:null-deep} and the assumed 
upper bound on movement give
\begin{equation}
  \Nphase\frac{\Gamma}{2\G}
  \le
  \sum_{t=0}^{\Nphase-1}\norm{\z_{t+1}-\z_t}_{\q}
    = \bigotilde{\RR \Nphase^{1-\rho}} 
\end{equation}
which implies
\begin{equation}\label{eq:phase-length-proof}
\Nphase = \bigotildel{\left(\frac{\G\RR}{\Gamma}\right)^{1/\rho}}.
\end{equation}
The terminal query, either a \serious{} step or a \nullstep{} step followed by
a successful stopping test in line~\ref{line:buffered-bundle-emptiness}, adds at most one call; the test itself
needs no first-order call.
If the test fails, its certificate gives
\[ \model(y)\le a+\Gamma/32<\lb_0+5\Gamma/32<\hminus,\]
so every queried body is nonempty.
It remains to verify that the certified gap contracts after the phase.  Let
$\lb$ and $\U$ denote the bounds at the start of the phase.  If a \serious{}
step terminates the phase, then the updated bounds satisfy
$\U_{\mathrm{new}}\le \hplus=\lb+3\Gamma/4$ and
$\lb_{\mathrm{new}}\ge\lb$, so
$\U_{\mathrm{new}}-\lb_{\mathrm{new}}\le3\Gamma/4$.
If instead the stopping test in line~\ref{line:buffered-bundle-emptiness} succeeds, then
$\lb_{\mathrm{new}}=\max\{\lb,a\}\ge\lb+\Gamma/8$, while
$\U=\lb+\Gamma$, and hence, since $\U_{\mathrm{new}}\le\U$,
$\U_{\mathrm{new}}-\lb_{\mathrm{new}}\le7\Gamma/8$.
The update is certified because $a\le\min_{\X}\model\le\fstar$.
Thus every phase contracts the certified gap by a factor of at most $7/8$.
A phase with gap $\Gamma$ therefore uses
$\bigOtilde((\G\RR/\Gamma)^{1/\rho})$ calls, including its terminal query.
The initialization $\lb=\f(\xquery_0)-2\G\RR$ is certified because
$\f$ is $\G$-Lipschitz in $\ell_{\q}$ and $\X$ has $\ell_{\q}$ diameter
at most $2\RR$ when ${\p\le\q}$.  Thus $\Gamma_0=2\G\RR$.
If $\eps\ge2\G\RR$, the initial certificate already gives the result.  Otherwise,
run \cref{alg:lipschitz-bundle} with target $\eps$.  Every active phase has
$\Gamma>\eps$, and there are at most $\bigO(\log(2\G\RR/\eps))$ phases.  Summing
    their costs and absorbing this logarithm into the $\bigotilde{\cdot}$ notation gives
$\bigOtilde((\G\RR/\eps)^{1/\rho})$ oracle calls.  At termination,
$\U-\lb\le\eps$ and $\U$ is the value of the
best queried point, so the returned point has error at most $\eps$.

Only transitions within a phase enter the \nullstep-step count.  The first
center of each phase adds one query.  If global query movement is also
recorded, transitions between phases cost at most $2\RR$ each in $\ell_{\q}$,
adding $\bigO(\RR\log(2\G\RR/\eps))$ in total.
\end{proof}

\begin{proof}\linkofproof{lem:buffered-implementation}
We start by noting that the lower-bound initialization $\lb\gets \f(\xquery_0)-2\G{}\RR$ works because the initial tangent has range at most $2\G\RR$.
Every subsequent lower-bound update uses $a\le\min_\X\model
\le \fstar$, while $\U$ keeps track of the best observed function value.  On the lower-bound
    branch in Line~\ref{line:buffered-bundle-emptiness}, the new gap is at most $7\Gamma/8$.  On a \serious{} step (Line~\ref{line:buffered-bundle-serious-update}) it is at most
$3\Gamma/4$, even without updating $\lb$ with the new first-order information.

The sets remain nested during a phase.
By \Cref{lem:null-deep}, if we have $\Nphase$ consecutive \nullstep{} steps in a phase between queried centers, so without the stopping test in Line~\ref{line:buffered-bundle-emptiness} triggering the end of the phase, the movement is
at least $\Nphase\Gamma/(2\G)$.  For the queried points $\zhat_t$
and reference centers $\z_t$, set %
$\Ectr=\max_{0\le t\le\Nphase}\norm{\zhat_t-\z_t}_{\q}
\le\Gamma/(16\G)$.  The triangle inequality then gives
\[
  \frac{\Nphase\Gamma}{2\G}
  \le\sum_{t<\Nphase}\norm{\z_{t+1}-\z_t}_{\q}+2\Nphase\Ectr
    \qquad \Longrightarrow \qquad
  \frac{3\Nphase\Gamma}{8\G}
  \le\sum_{t<\Nphase}\norm{\z_{t+1}-\z_t}_{\q}.
\]
    The movement bound $\bigotilde{\RR{}N^{1-\rho}}$ of the reference centers $\z_t$ in \cref{eq:center-movement} therefore gives the phase-length estimate $N = \bigotilde{(\G{}\RR / \Gamma)^{1 / \rho}}$, similarly as in %
    the derivation of \eqref{eq:phase-length-proof} in the proof of \cref{thm:movement-reduction}.
To sum the per-phase bound in \cref{eq:phase-length-proof}, let
$\Gamma_0,\ldots,\Gamma_{J-1}>\eps$ be the active phase gaps.
Since $\Gamma_{k+1}\le(7/8)\Gamma_k$,
\[
  \sum_{k=0}^{J-1}\left(\frac{\G\RR}{\Gamma_k}\right)^{1/\rho}
  \le
  \left(\frac{\G\RR}{\eps}\right)^{1/\rho}
  \sum_{j=0}^{J-1}(7/8)^{j/\rho}
  \le
  \frac{(\G\RR/\eps)^{1/\rho}}{1-(7/8)^{1/\rho}}\le 8\left(\frac{\G\RR}{\eps}\right)^{1/\rho}.
\]
    The $J=\bigO(1+\log(1+2\G\RR/\eps))$ terminal and initial-center calls
are absorbed by the logarithmic factors, giving the oracle bound in
\eqref{eq:oracle-compl-from-movement}. 
Beyond the transitions counted above, there are at most two additional first-order calls per phase (its initial-center and terminal calls), plus the single initialization call, hence at most $2J+1$ additional calls in total. If $\eps<2\G\RR$, then $J\le\lceil\log(2\G\RR/\eps)/\log(8/7)\rceil$; otherwise $J=0$. This overhead is $\bigO(1+(\G\RR/\eps)^{1/\rho})$ for $0<\rho\le1$, giving \eqref{eq:oracle-compl-from-movement}.

\end{proof}

\section{Proofs from \Cref{sec:general-center}}\label{sec:value_budget}

\begin{proof}\linkofproof{lem:head-tail}
If $\p=\infty$, then $\stail=\infty$, $u=0$, and $v=x$, so the claims hold.
Assume henceforth that $\p<\infty$.
For a head coordinate, $\abs{x_i}>\RR\tau$ and $1-\p\le0$, hence
\begin{equation*}
  \abs{x_i}
  =\abs{x_i}^{\p}\abs{x_i}^{1-\p}
  \le (\RR\tau)^{1-\p}\abs{x_i}^{\p}.
\end{equation*}
For a tail coordinate and finite $\stail\ge \p$,
\begin{equation*}
  \abs{x_i}^{\stail}
  =\abs{x_i}^{\p}\abs{x_i}^{\stail-\p}
  \le (\RR\tau)^{\stail-\p}\abs{x_i}^{\p}.
\end{equation*}
Since $\sum_i\abs{x_i}^{\p}\le \RR^{\p}$ and $\tau=\T^{-1/\p}$, summation gives
\begin{equation*}
  \norm{u}_1\le \RR\tau^{1-\p}=\RR\T^{1-1/\p},
  \qquad
  \norm{v}_{\stail}\le \RR\tau^{1-\p/\stail}=\RR\T^{1/\stail-1/\p}.
\end{equation*}
These are exactly $\Rhead$ and $\Rtail_{\stail}$.
For $\stail=\infty$, the tail bound follows directly from
$\norm{v}_\infty\le \RR\tau=\RR\T^{-1/\p}$.
\end{proof}

\subsection{Proof of the unified \texorpdfstring{low-$\p$}{low-p} movement bound}

For the program in \cref{eq:unified-sample-program}, abbreviate
\begin{equation}\label{eq:sample-values}
  \usample_t(\gdir)=\usample_{\K_t}(\gdir),
  \qquad
  \vsample_t(\gdir)=\vsample_{\K_t}(\gdir),
  \qquad
  \newtarget{def:sample-value}{\mval_t}(\gdir)=\F_{\gdir}(\usample_t(\gdir),\vsample_t(\gdir)).
\end{equation}
The increment in the optimal value is
\begin{equation}\label{eq:value-increment}
  \newtarget{def:value-increment}{\Delta_t}(\gdir)=\mval_{t+1}(\gdir)-\mval_t(\gdir).
\end{equation}

\begin{lemma}[Unified charge and value budget]
\label{lem:unified-charge-budget}
For nested bodies $\K_0\supseteq\cdots\supseteq\K_{\T}$ in
$\RR\ball{\p}{\d}$, use the sample optimizers and value increments in
\cref{eq:sample-values,eq:value-increment}, with parameters from
\cref{eq:unified-exponents,eq:gaussian-parameters}.
There is a universal constant $c>0$ such that, for every $\gdir\in\R^{\d}$,
\begin{equation*}
  \frac{\rmin-1}{2}
  \sum_{t=0}^{\T-1}\norm{\vsample_{t+1}(\gdir)-\vsample_t(\gdir)}_{\rmin}^2
  +\frac{c\mu}{\Ld}
  \sum_{t=0}^{\T-1}\norm{\usample_{t+1}(\gdir)-\usample_t(\gdir)}_1^2
  \le
  \sum_{t=0}^{\T-1}\Delta_t(\gdir)
  \le
  \Rtail_{\rmin}^2+2\eta \Rhead\norm{\gdir}_\infty.
\end{equation*}
Moreover,
\begin{equation*}
  \E\sum_{t=0}^{\T-1}\Delta_t(\Z) \le
  \left(
    \E\left(\sum_{t=0}^{\T-1}\Delta_t(\Z)\right)^2
  \right)^{1/2}
  \le C \Rtail_{\rmin}^2.
\end{equation*}

\end{lemma}

\begin{proof}
Write $w_t=(\usample_t(\gdir),\vsample_t(\gdir))$.  The Bregman divergence of $\F_{\gdir}$ is the sum
of the divergences of its two squared norm terms; the linear perturbation
contributes nothing.  The lower and upper bounds follow from the single
chain
\begin{equation*}
  \begin{aligned}
    \frac{\rmin-1}{2}
      \sum_{t=0}^{\T-1}&\norm{\vsample_{t+1}(\gdir)-\vsample_t(\gdir)}_{\rmin}^2
    +\frac{c\mu}{\Ld}
    \sum_{t=0}^{\T-1}\norm{\usample_{t+1}(\gdir)-\usample_t(\gdir)}_1^2 \\
    &\circled{1}[\le]
    \sum_{t=0}^{\T-1}\D_{\F_{\gdir}}(w_{t+1},w_t)
    \circled{2}[\le]
    \sum_{t=0}^{\T-1}\Delta_t(\gdir) =\F_{\gdir}(w_{\T})-\F_{\gdir}(w_0) \\
    &\circled{3}[\le]
    \sup_{w\in\Lift_{\p,\q,\T}(\K_0)}\F_{\gdir}(w)
    -\inf_{w\in\Lift_{\p,\q,\T}(\K_0)}\F_{\gdir}(w) \\
    &\circled{4}[\le]
    \frac{\Rtail_{\rmin}^2}{2}
    +\frac{\mu \Rhead[2]}{2}
    +2\eta \Rhead\norm{\gdir}_\infty
    =\Rtail_{\rmin}^2+2\eta \Rhead\norm{\gdir}_\infty.
  \end{aligned}
\end{equation*}
Here \circled{1} uses item 1 in \cref{fact:unif-cvx-ell-p} %
and \cref{lem:near-l1-geometry}.  For \circled{2}, the
Bregman identity expresses $\Delta_t(\gdir)$ as
$\D_{\F_{\gdir}}(w_{t+1},w_t)+\ip{\nabla \F_{\gdir}(w_t)}{w_{t+1}-w_t}$; nestedness gives
$w_{t+1}\in\Lift_{\p,\q,\T}(\K_{t+1})\subseteq\Lift_{\p,\q,\T}(\K_t)$, so first-order
optimality of $w_t$ gives
$\ip{\nabla \F_{\gdir}(w_t)}{w_{t+1}-w_t}\ge0$.  The inequality \circled{3} uses
$w_0,w_{\T}\in\Lift_{\p,\q,\T}(\K_0)$.  Finally, \circled{4} uses
$0 \le \norm{v}_{\rmin}\le \Rtail_{\rmin}$, $0 \le\norm{u}_{\snear}\le \Rhead$, and
$\abs{\ip{\gdir}{u}}\le \Rhead\norm{\gdir}_\infty$; the last identity uses
$\mu \Rhead[2]=\Rtail_{\rmin}^2$.  The same first-order argument before summing gives
$\Delta_t(\gdir)\ge \D_{\F_{\gdir}}(w_{t+1},w_t)\ge0$ for every $t<\T$.
This proves the first part. 

Finally,
\begin{equation*}
  \E\sum_{t=0}^{\T-1}\Delta_t(\Z)
  \circled{1}[\le]
  \left(
    \E\left(\sum_{t=0}^{\T-1}\Delta_t(\Z)\right)^2
  \right)^{1/2}
  \circled{2}[\le]
  \Big(2(\Rtail_{\rmin}^2)^2
  +2(2\eta \Rhead)^2 \E\norm{\Z}_\infty^2\Big)^{1/2}
  \circled{3}[\le] C \Rtail_{\rmin}^2.
\end{equation*}
Here \circled{1} follows from the Jensen inequality, and \circled{2}  uses the above derivation, nonnegativity of $\Delta_t$ and Young's inequality, and \circled{3} uses the Gaussian maximum  estimate $\E\norm{\Z}_{\infty}^2\le c\log(2\d)$, together with \eqref{eq:gaussian-parameters}. %
\end{proof}

\begin{proof}\linkofproof{prop:low-p-center}
Let $\newtarget{def:mean-components}{\ubar_t}=\E \usample_t(\Z)$ and $\vbar_t=\E \vsample_t(\Z)$.  We first bound the tail
and the head in the two norms used at the final interpolation step. Let $\rmin:=\min\{\q,2\}$.

\paragraph{Tail movement.}\leavevmode
Cauchy-Schwarz, Jensen's inequality for the squared norm,
and \cref{lem:unified-charge-budget} give
\begin{align}
  \sum_{t=0}^{\T-1}\norm{\vbar_{t+1}-\vbar_t}_{\rmin}
  \le
  \sqrt{
    \T\E\sum_{t=0}^{\T-1}
    \norm{\vsample_{t+1}(\Z)-\vsample_t(\Z)}_{\rmin}^2
  }
  \le
  \sqrt{
    \frac{2\T}{\rmin-1}
    \E\sum_{t=0}^{\T-1}\Delta_t(\Z)
  }
  \le
  C_{\rmin}\Rtail_{\rmin}\sqrt \T.
  \label{eq:unified-tail-movement}
\end{align}

\paragraph{Head movement in $\ell_1$.}
Cauchy--Schwarz and Jensen, applied as in
\cref{eq:unified-tail-movement} to the head term in
\cref{lem:unified-charge-budget}, give
\begin{equation} \label{eq:unified-head-l1}
  \sum_{t=0}^{\T-1}\norm{\ubar_{t+1}-\ubar_t}_1
  \le
  \sqrt{
    \frac{\T \Ld}{c\mu}
    \E\sum_{t=0}^{\T-1}\Delta_t(\Z)
  }
  \le
  C \Rhead\sqrt{\T \Ld}.
\end{equation}
Here $\mu=\Rtail_{\rmin}^2/\Rhead[2]$ by \cref{eq:gaussian-parameters}.

\paragraph{Head movement in $\ell_2$.}
In \cref{eq:unified-sample-program}, the only dependence of $\F_{\gdir}$ on $\gdir$ is the linear term
$-\eta\ip{\gdir}{u}$.  For the optimal value $\mval_t$ in
\cref{eq:sample-values}, uniqueness of the sample optimizer and Danskin's theorem
give
$
  \nabla \mval_t(\gdir)=-\eta \usample_t(\gdir).
$
Since $\norm{\usample_t(\gdir)}_2\le \Rhead$, each $\mval_t$ is
$\eta \Rhead$-Lipschitz.  Thus $\Delta_t$ is Lipschitz, and Gaussian
integration by parts applies coordinatewise, yielding
\begin{equation}\label{eq:euclidean-ibp}
\begin{aligned}
  \ubar_{t+1}-\ubar_t
  =\E[\usample_{t+1}(\Z)-\usample_t(\Z)]
  =
  -\frac1\eta\E[\nabla\Delta_t(\Z)]
  =-\frac1\eta\E[\Z\Delta_t(\Z)].
\end{aligned}
\end{equation}
Let $b_t=\E[\Z\Delta_t(\Z)]$ and take the deterministic unit vector
$a_t=b_t/\norm{b_t}_2$ when $b_t\ne0$, and any deterministic Euclidean unit
vector when $b_t=0$.  Since every $\Delta_t$ is nonnegative,
\begin{align} \label{eq:unified-head-l2}
    \begin{aligned}
  \sum_{t=0}^{\T-1}\norm{\ubar_{t+1}-\ubar_t}_2
  &\circled{1}[\le]
  \frac1\eta
  \E\left[
    \left(\sum_{t=0}^{\T-1}\Delta_t(\Z)\right)
    \max_{t<\T}\abs{\ip{\Z}{a_t}}
  \right] \\
  &\circled{2}[\le]
  \frac1\eta
  \left(
    \E\left(\sum_{t=0}^{\T-1}\Delta_t(\Z)\right)^2
  \right)^{1/2}
  \left(
    \E\max_{t<\T}\abs{\ip{\Z}{a_t}}^2
  \right)^{1/2} \\
  &\circled{3}[\le]
  C\frac{\Rtail_{\rmin}^2}{\eta}\sqrt{\log(2\T)}
  =
  C \Rhead\sqrt{\Ld\log(2\T)}.
    \end{aligned}
\end{align}
Here $\circled{1}$ uses \cref{eq:euclidean-ibp} and the choice of $a_t$.
The inequality $\circled{2}$ is Cauchy-Schwarz.  The inequality
$\circled{3}$ uses \cref{lem:unified-charge-budget} and the Gaussian maximum
estimate 
\(
    \sqrt{\E\max_{t<\T}\abs{\ip{\Z}{a_t}}^2}
  \le C\sqrt{\log(2\T)}.
\)

\paragraph{Common conversion to $\ell_{\q}$.}
Let
 \( \theta_{\q}=\left(\frac2\q-1\right)_{\pospart}\). 
If $\q<2$, by the convexity of $\p\mapsto\log\norm{x}_{\p}^{\p}$,%
\begin{equation*}
  \norm{x}_{\q}
  \le
  \norm{x}_1^{\theta_{\q}}\norm{x}_2^{1-\theta_{\q}}.
\end{equation*}
If $\q\ge2$, then $\theta_{\q}=0$ and use norm monotonicity to conclude that
$\norm{x}_{\q}\le\norm{x}_2$.  Applying H\"older's inequality  and 
using \cref{eq:unified-head-l1,eq:unified-head-l2} gives
\begin{align}
  \sum_{t=0}^{\T-1}\norm{\ubar_{t+1}-\ubar_t}_{\q}
  &\le
  \left(
    \sum_{t=0}^{\T-1}\norm{\ubar_{t+1}-\ubar_t}_1
  \right)^{\theta_{\q}}
  \left(
    \sum_{t=0}^{\T-1}\norm{\ubar_{t+1}-\ubar_t}_2
  \right)^{1-\theta_{\q}}\notag\\
  &\le
  C_{\q}\Rhead
  \T^{1/\rmin-1/2}
  \sqrt{\Ld}\,
  (\log(2\T))^{1-1/\rmin},
  \label{eq:unified-head-lq}
\end{align}
where we used that $\frac{\theta_{\q}}{2}=\frac1\rmin-\frac12$ and $\frac{1-\theta_{\q}}{2}=1-\frac1\rmin$.
Finally, $\q\ge \rmin$ implies
$\norm{\vbar_{t+1}-\vbar_t}_{\q}
\le\norm{\vbar_{t+1}-\vbar_t}_{\rmin}$.  Combining
\cref{eq:unified-tail-movement,eq:unified-head-lq} with
\begin{equation*}
  \Rtail_{\rmin}\sqrt \T
  =
  \Rhead\T^{1/\rmin-1/2}
  =
  \RR\T^{1-\frac1\p+\left(\frac1\q-\frac12\right)_{\pospart}}
\end{equation*}
gives
\begin{equation*}
  \sum_{t=0}^{\T-1}
  \norm{\Cgauss(\K_{t+1})-\Cgauss(\K_t)}_{\q}
  \le
  C_{\q}\RR
  \T^{1-\frac1\p+\left(\frac1\q-\frac12\right)_{\pospart}}
  \left[
    1+\sqrt{\Ld}\,(\log(2\T))^{1-1/\rmin}
  \right].
\end{equation*}
This proves the theorem.
\end{proof}

\subsection{Proof of the energy selector}

\begin{proof}\linkofproof{prop:energy-complement}
Write $r=\max\{a,2\}$.  Then
$\Psipq(x)=\norm{x}_a^r/r$, and norm comparison gives
$\K_t\subseteq\RR\d^{(1/a-1/\p)_{\pospart}}\ball{a}{\d}$ for every $0\le t\le\T$.
By \cref{fact:unif-cvx-ell-p,prop:Bregman-LB-unif-cvx}, this
regularizer satisfies
\[
  \begin{aligned}
    \Psipq(y)-\Psipq(x)
      -\ip{\nabla\Psipq(x)}{y-x}
    \ge\frac{\kappa_a}{r}\norm{y-x}_a^r,
    \qquad
    \kappa_a=
    \begin{cases}
      a-1, & 1<a\le2,\\
      2^{2-a}, & 2<a<\infty.
    \end{cases}
  \end{aligned}
\]
Uniform convexity and constrained optimality therefore give
\[
  \begin{aligned}
  \frac{\kappa_a}{r\T^{r-1}}
  \left(\sum_{t=0}^{\T-1}\norm{\z_{t+1}-\z_t}_a\right)^r
  &\circled{1}[\le]
    \frac{\kappa_a}{r}
    \sum_{t=0}^{\T-1}\norm{\z_{t+1}-\z_t}_a^r \\
    &\circled{2}[\le]
    \sum_{t=0}^{\T-1}\Psipq(\z_{t+1})-\Psipq(\z_t)
      -\ip{\nabla\Psipq(\z_t)}{\z_{t+1}-\z_t} \\
    &\circled{3}[\le]
    \sum_{t=0}^{\T-1}\bigl(\Psipq(\z_{t+1})-\Psipq(\z_t)\bigr)\\
    &=\Psipq(\z_{\T})-\Psipq(\z_0)
    \le\frac{\RR^r\d^{r(1/a-1/\p)_{\pospart}}}{r}.
  \end{aligned}
\]
Above, $\circled{1}$ uses Hölder's inequality while $\circled{2}$ uses the Bregman lower bound above.  For $\circled{3}$, $\z_{t+1} \in \K_t$ by nestedness so by optimality of the definition of $\z_t$, we have
$\ip{\nabla\Psipq(\z_t)}{\z_{t+1}-\z_t}\ge0$.
Reorganizing and converting the movement to the $\q$ norm gives
\begin{equation}\label{eq:energy-complement-conversion}
  \begin{aligned}
    \sum_{t=0}^{\T-1}\norm{\z_{t+1}-\z_t}_{\q}
    &\le\d^{(1/\q-1/a)_{\pospart}}
      \sum_{t=0}^{\T-1}\norm{\z_{t+1}-\z_t}_a\\
    &\le\kappa_a^{-1/r}\RR
      \d^{(1/\q-1/a)_{\pospart}+(1/a-1/\p)_{\pospart}}
      \T^{1-1/r}.
  \end{aligned}
\end{equation}
For $\q>1$ and $(\p,\q)\ne(\infty,\infty)$, the chosen $a$ lies
between $\p$ and $\q$, so the dimension exponent is
$(1/\q-1/\p)_{\pospart}$.
Moreover, $r=\p$ when $2\le\p<\q$, and
$r=\max\{\q,2\}$ when $\p\ge\q$.
If $a<2$, then $a=\p$ and
$\kappa_a^{-1/r}=(\p-1)^{-1/2}$; if $a\ge2$, then
$\kappa_a^{-1/r}\le2$.
Substitution into \cref{eq:energy-complement-conversion} proves
\cref{eq:energy-complement-movement} in these cases.

For $\q=1$, let $s=1+1/\Ld$.  The equivalence
$\norm{x}_1=\Theta(\norm{x}_s)$, with universal constants,
and $\K_t\subseteq\RR\d^{1-1/\p}\ball{1}{\d}$ allow us to apply
the energy-norm movement bound with $a=s$, giving
$O(\RR\d^{1-1/\p}\sqrt{\Ld\T})$ movement in $\ell_1$.
Similarly, $\norm{x}_{\infty}=\Theta(\norm{x}_{\Ld})$ with
universal constants, so the bound with $a=\Ld$ gives $O(\RR\T)$
when $\p=\q=\infty$.
These are the claimed bounds after absorbing $\sqrt{\Ld}$ into
$\bigOtilde$ at $\q=1$.
The multiplicative constants may be taken to be
$\max\{2e,(\p-1)^{-1/2}\}$ for $1<\p<2$ and $2e$ for
$\p=1$ or $2\le\p\le\infty$.
Since the regularizer depends only on $\p,\q$ and $\d$, the selector
is independent of the horizon $\T$.
\end{proof}

One could also have defined the algorithm analyzed in \cref{prop:energy-complement} as performing a Bregman projection to the new set from the previous iterate: $\z_{t+1} = \argmin_{z \in \K_{t+1}} \D_{\Psipq}(z,\z_t)$. An almost identical analysis gives the same movement bound.

\section{Proofs from \Cref{sec:randomized-implementation}}

\begin{proof}\linkofproof{prop:monte-carlo}
Fix a round $t$ and condition on the history before its samples are drawn.
Write $\K=\K_t$, $\Y_j=\Y_{t,j}$, and $\Yhat_j=\Yhat_{t,j}$.
The body $\K$ is now fixed and the fresh samples are independent.
Let $C=\Cgauss(\K)=\E[\usample_{\K}(\Z)+\vsample_{\K}(\Z)]$ and define
the exact empirical average
\[
  \newtarget{def:empirical-center}{\Chat_{\Nsample}}(\K)
  =\frac1\Nsample\sum_{j=1}^{\Nsample}\Y_j.
\]
We first bound this average's error, then transfer the bound to numerical
samples and sum the movement errors on one simultaneous success event.
\paragraph{Exact empirical average.}
Every $\Y_j$ belongs to the convex set $\K$, so $\Chat_{\Nsample}(\K)\in \K$.
Moreover, $\K\subseteq \RR \ball{\p}{\d}$ and ${\p\le\rmin}$, hence
$\norm{\Y_j}_{\rmin}\le\norm{\Y_j}_{\p}\le \RR$.

Let $\sigma_1,\ldots,\sigma_{\Nsample}$ be independent Rademacher signs, independent
of the samples.  We follow the standard symmetrization argument followed by the type-$\rmin$
inequality for $\ell_{\rmin}$,%
\begin{equation*}
  \begin{aligned}
    \E\norm{\Chat_{\Nsample}(\K)-C}_{\rmin}
    \circled{1}[\le]
    \frac2\Nsample\E\norm{\sum_{j=1}^{\Nsample}\sigma_j\Y_j}_{\rmin} 
    \circled{2}[\le]
    \frac2\Nsample\left(\E\sum_{j=1}^{\Nsample}\norm{\Y_j}_{\rmin}^{\rmin}\right)^{1/\rmin}
    \le 2\RR \Nsample^{-1/\rmin^{\conj}}.
  \end{aligned}
\end{equation*}
Here $\circled{1}$ is the symmetrization inequality and $\circled{2}$ is the
Rademacher type inequality; see \citet{ledouxTalagrand1991} for both facts.
For completeness, the second step follows coordinatewise from
\begin{equation*}
  \E_\sigma\norm{\sum_{j=1}^{\Nsample}\sigma_jy_j}_{\rmin}^{\rmin}
  =
  \sum_{i=1}^{\d}\E_\sigma\abs{\sum_{j=1}^{\Nsample}\sigma_jy_{j,i}}^{\rmin}
  \circled{1}[\le]
  \sum_{i=1}^{\d}\left(\sum_{j=1}^{\Nsample}y_{j,i}^2\right)^{\rmin/2}
  \circled{2}[\le]
  \sum_{j=1}^{\Nsample}\norm{y_j}_{\rmin}^{\rmin}.
\end{equation*}
Here $\circled{1}$ is the upper Khintchine inequality, which holds for all
$r\in(0,\infty)$ with an $r$-dependent constant and here the constant is one
since $\rmin\le2$.  Step $\circled{2}$ uses $\rmin\le2$ and norm monotonicity,
$\|a\|_2\le\|a\|_{\rmin}$, applied to each coordinate sequence
$a=(y_{j,i})_{j=1}^{\Nsample}$.

Replacing one sample $\Y_j$ by $\Y_j'$ changes the norm of the centered
empirical average by at most
\begin{equation*}
  \frac1\Nsample\norm{\Y_j-\Y_j'}_{\rmin}\le\frac{2\RR}{\Nsample}.
\end{equation*}
The bounded-differences inequality of \citet{mcdiarmid1989bounded} therefore
gives, conditionally on the history, with probability at least
$1-\delta/(\T+1)$,
\begin{equation}\label{eq:monte-carlo-single-error}
  \norm{\Chat_{\Nsample}(\K)-C}_{\rmin}
  \le
  \frac{2\RR}{ \Nsample^{1/\rmin^{\conj}}}
  +\RR\sqrt{\frac{2\log((\T+1)/\delta)}{\Nsample}}.
\end{equation}
The sample-size condition in \cref{eq:monte-carlo-sample-size} makes the two terms on the right each at most $\errE/2$ after
enlarging $C_{\rmin}$.  Since $\q\ge \rmin$, norm monotonicity gives the claimed
$\ell_{\q}$ bound.

\paragraph{Numerical solutions.}
If $\Yhat_j\in \K$ and
$\norm{\Yhat_j-\Y_j}_{\q}\le \Esol$, convexity preserves
feasibility and, on the single-round success event
$\norm{\Chat_{\Nsample}(\K)-C}_{\q}\le\errE$ from
\cref{eq:monte-carlo-single-error} and the sample-size condition
\cref{eq:monte-carlo-sample-size},
\begin{equation*}
  \norm{
    \frac1\Nsample\sum_{j=1}^{\Nsample}\Yhat_j-C
  }_{\q}
  \le
  \norm{\Chat_{\Nsample}(\K)-C}_{\q}
  +\frac1\Nsample\sum_{j=1}^{\Nsample}\norm{\Yhat_j-\Y_j}_{\q}
  \le \errE+\Esol.
\end{equation*}

We next explain how an objective tolerance ensures the assumed samplewise
error.  Fix $\gdir$, let $w=(u,v)$ be the exact optimizer in
\cref{eq:unified-sample-program}, and let
$\widehat w=(\widehat u,\widehat v)$ be feasible with
$\F_{\gdir}(\widehat w)-\F_{\gdir}(w)\le\newtarget{def:objective-tolerance}{\zeta}$.  Then, with $c_0=e^{-2}/2$,%
\begin{equation*}
  \begin{aligned}
    \zeta
    &\ge
    \D_{\F_{\gdir}}(\widehat w,w)
    +\ip{\nabla \F_{\gdir}(w)}{\widehat w-w} \circled{1}[\ge]
    \D_{\F_{\gdir}}(\widehat w,w) \circled{2}[\ge]
    \frac{c_0\mu}{\Ld}\norm{\widehat u-u}_1^2
    +\frac{\rmin-1}{2}\norm{\widehat v-v}_{\rmin}^2.
  \end{aligned}
\end{equation*}
The inequality $\circled{1}$ is first-order optimality of $w$ over the convex
lift.  For $\circled{2}$, \cref{lem:near-l1-geometry} gives the head term, and \cref{fact:unif-cvx-ell-p} gives the tail term. 
Thus, if we take
\begin{equation}\label{eq:sample-objective-tolerance}
  \zeta
  \le
\Esol[2]\min\left\{\frac{c_0\mu}{4\Ld},\frac{\rmin-1}{8}\right\},
\end{equation}
then since $\q\ge \rmin$, we have
\begin{equation*}
  \norm{(\widehat u+\widehat v)-(u+v)}_{\q}
  \le
  \norm{\widehat u-u}_1+\norm{\widehat v-v}_{\rmin} \leq \sqrt{\frac{\Ld\zeta}{c_0\mu}} + \sqrt{\frac{2\zeta}{\rmin-1}} \leq \Esol.
\end{equation*}

\paragraph{Path error.}
Each of the $\T+1$ conditional failure probabilities is at most
$\delta/(\T+1)$.  Taking expectations over the histories and applying a
union bound gives simultaneous success with probability at least $1-\delta$,
even for adaptively chosen bodies.  On this event, every queried average
is within $\errE+\Esol$ of $\z_t=\Cgauss(\K_t)$.

More generally, write
$\newtarget{def:center-error}{\Ectr}
=\max_{0\le t\le\T}\norm{\zhat_t-\z_t}_{\q}$.
The triangle inequality gives the deterministic path-error bound
\begin{equation}\label{eq:approximate-path-error}
  \begin{aligned}
    \sum_{t=0}^{\T-1}\norm{\zhat_{t+1}-\zhat_t}_{\q}
    &\le
    \sum_{t=0}^{\T-1}\left(
      \norm{\z_{t+1}-\z_t}_{\q}
      +\norm{\zhat_{t+1}-\z_{t+1}}_{\q}
      +\norm{\zhat_t-\z_t}_{\q}
    \right) \\
    &\le
    \sum_{t=0}^{\T-1}\norm{\z_{t+1}-\z_t}_{\q}
    +2\T\Ectr.
  \end{aligned}
\end{equation}
On the simultaneous success event, $\Ectr\le\errE+\Esol$.
For the choices in the proposition,
\[
  \Ectr
  \le \frac{\RR}{2}
  \T^{-\frac1\p+(\frac1\q-\frac12)_{\pospart}}.
\]
Substituting this bound into \cref{eq:approximate-path-error} and using the
exact-center movement bound in \cref{prop:low-p-center} proves
\cref{eq:monte-carlo-movement}.
\end{proof}

\begin{proof}\linkofproof{thm:polynomial-implementation}
Rescaling to $\f(\RR x)/(\G\RR)$ with tolerance $\eps/(\G\RR)$ lets us
assume $\G=\RR=1$.

Replace irrational finite exponents by rational $\p'\le\p$ and $\q'\ge\q$
sufficiently close that norm comparisons and complexity bounds change only
by absolute factors; keep $\q=\infty$ unchanged.

Choose $\p'$ so that $\ball{\p'}{\d}\subseteq\ball{\p}{\d}
\subseteq(1+\eps/4)\ball{\p'}{\d}$: solving on the inner ball to accuracy
$\eps/2$ and subtracting $\eps/4$ from its lower certificate preserves
feasibility and $\eps$-accuracy for the original problem.
Absorb these constants and relabel the rational parameters below.
We first construct the model-minimum certificate on
the original ball, then use its strict model-level margin to exhibit interior balls
for the queried sets and lifts.  All subproblem objectives and constraints
use only explicit norm functions and stored oracle cuts; solving them
requires no additional calls to the first-order oracle of $\f$.

\paragraph{The model-minimum certificate.}
The original domain has explicit inner and outer balls:
\[
  \d^{-1/2}\ball{2}{\d}
  \subseteq\X=\ball{\p}{\d}
  \subseteq\d^{1/2}\ball{2}{\d}.
\]
The model $\model$ is convex and $1$-Lipschitz in $\ell_{\q}$, since each
stored cut has normal of $\ell_{\q^{\conj}}$ norm at most one.  Its value
and a subgradient are computed by finding an active affine cut.
Norm evaluation supplies a strong separation oracle for $\X$.
The separation-based ellipsoid method therefore produces a feasible $y$
and certified bounds
$a\le\min_{\X}\model\le\model(y)\le b$ with $b-a\le\Gamma/32$ in polynomial
time; see \citep[Chapter~4]{groetschelLovaszSchrijver1993} and the
objective-level argument below.  This constructs 
line \ref{line:buffered-certified-near-minimum-model} of \Cref{alg:lipschitz-bundle} without an assumption on the geometry of any model sublevel set.

\paragraph{Interior balls for model sublevel sets.}%
Before each center query, the stopping test in
Line~\ref{line:buffered-bundle-emptiness} has failed.
The certificate in Line~\ref{line:buffered-certified-near-minimum-model}
therefore provides a point $y\in\X$ satisfying
\[
  \model(y)-\min_{x\in\X}\model(x)\le b-a\le\Gamma/32.
\]
Moreover, failure of the stopping test means $a<\lb_0+\Gamma/8$,
where $\lb_0$ is the lower bound frozen at the start of the phase. Hence
\begin{equation}\label{eq:near-minimizer-margin}
  \model(y)\le b\le a+\Gamma/32
  <\lb_0+5\Gamma/32=\hminus-3\Gamma/32.
\end{equation}
Thus $\K$ is nonempty, and $y$ has a strict model-level margin.
Set $\alpha=\Gamma/64$ and $c=(1-\alpha)y$.  The initialization and valid
certificate updates give $0<\Gamma\le2$, so $0<\alpha<1$.
Since $\norm{y}_{\q}\le\norm{y}_{\p}\le1$,
\[
  \norm{c}_{\p}\le1-\alpha,
  \qquad
  \model(c)\le\model(y)+\alpha
  <\hminus-5\alpha.
\]
For any displacement $h$,
$\norm{h}_{\p},\norm{h}_{\q}\le\sqrt{\d}\norm{h}_2$.
Consequently,
\[
  c+\frac{\Gamma}{128\sqrt{\d}}\ball{2}{\d}
  \subseteq\K=\{x\in\X:\model(x)\le\hminus\}
  \subseteq c+2\sqrt{\d}\ball{2}{\d}.
\]
Every model sublevel set %
is therefore full-dimensional, with a known interior
point and inverse-polynomial inner radius because $\Gamma>\eps$.

\paragraph{Interior balls for the lifts.}
Let $\rmin=\min\{\q,2\}$ and let $\T$ be the current selector
horizon. Let $y$ be as in line \ref{line:buffered-certified-near-minimum-model} of \cref{alg:lipschitz-bundle}. Apply the explicit decomposition in \cref{lem:head-tail} to it:
$y=u+v$, $\norm{u}_1\le\Rhead$, and
$\norm{v}_{\rmin}\le\Rtail_{\rmin}$.
The pair $((1-\alpha)u,(1-\alpha)v)$ has respective norm slacks
$\alpha\Rhead$ and $\alpha\Rtail_{\rmin}$, and its sum is $c$.
It follows that
\[
  ((1-\alpha)u,(1-\alpha)v)
  +\frac{\alpha\min\{1,\Rhead,\Rtail_{\rmin}\}}{4\sqrt{\d}}\ball{2}{2\d}
  \subseteq\Lift_{\p,\q,\T}(\K).
\]
Indeed, a product-space displacement $(h,k)$ of norm at most the displayed
radius changes either component norm by at most
$\sqrt{\d}\norm{(h,k)}_2$ and changes the sum in either $\ell_{\p}$ or
$\ell_{\q}$ by at most $\sqrt{2\d}\norm{(h,k)}_2$.
These changes are smaller than the component slacks and the slacks at $c$.
Moreover, $\T^{-1}\le\Rtail_{\rmin}\le1$ and $1\le\Rhead\le\T$.
Thus the displayed radius is at least
$\Gamma/(256\sqrt{\d}\T)$, and an outer radius about its center is at most
$2\sqrt{\Rhead[2]+\Rtail_{\rmin}^2}\le4\T$.
The lift is full-dimensional with explicit polynomial geometric bounds.

\paragraph{Feasible inexact convex solves.}
Both $\K$ and its lift have strong separation oracles obtained by checking
their norm constraints and stored affine inequalities.
The objective $\F_{\gdir}$ in \cref{eq:unified-sample-program} is convex.
Its values and subgradients use
polynomially many of the permitted arithmetic and rational-power
operations.  On the enclosing ball, its Euclidean Lipschitz bound is
polynomial in $\d$, the horizon, and $1+\norm{\gdir}_{\infty}$.
For completeness, suppose a convex objective $H$ with Euclidean Lipschitz
bound $L_H$ is minimized over a convex set $S\subseteq\R^n$ with a strong separation
oracle and known balls
$c+a\ball{2}{n}\subseteq S\subseteq c+b\ball{2}{n}$.
For a minimizer $w$, the point $w_\lambda=(1-\lambda)w+\lambda c$ has a
ball of radius $\lambda a$ in $S$.  Given objective tolerance $\xi>0$,
choose
\[
  \lambda=\min\{1/2,\xi/[8(L_H+1)b]\}.
\]
Then $H(w_\lambda)\le\min_S H+\xi/8$, and the ball about $w_\lambda$ of
radius $\min\{\lambda a/2,\xi/[8(L_H+1)]\}$ consists of feasible points
with objective value at most $\min_S H+\xi/4$.
An ellipsoid feasibility test at objective level $v$ therefore either
returns a feasible point of value at most $v$, or, after the corresponding
volume bound is reached, certifies $\min_S H>v-\xi/4$.
Bisection using these valid lower and upper bounds yields a feasible
point and a certified interval of width $\xi$ in polynomial time.
Subtracting $H(c)$ bounds the initial objective interval by $2L_Hb$.
The explicit inner and outer radii above make this argument polynomial
in the dimension, horizon, and inverse accuracy.

For a Gaussian sample, apply this procedure to $\F_{\gdir}$ with the
objective tolerance in \cref{eq:sample-objective-tolerance}.
Strong convexity then gives the required $\ell_{\q}$ error $\Esol$ from
the exact sample, while preserving feasibility.
For fixed norm parameters, the inverse of this sufficient objective
tolerance is polynomial in the horizon and prescribed inverse center
accuracy.
\paragraph{A fixed horizon in each phase.}
Set $\rho=1/\p-(1/\q-1/2)_{\pospart}$ and
$r=\min\{\q,2\}$. The proof of \cref{prop:low-p-center} gives an
explicit constant $C_{\q}$ such that the selector with planned horizon
$H\ge2$ has movement at most
\begin{equation}\label{eq:fixed-phase-horizon-budget}
  B(H):=C_{\q}\RR H^{1-\rho}
  \left[1+\sqrt{\Ld}\,(\log(2H))^{1-1/r}\right]
\end{equation}
over $H$ transitions. At the start of a phase with gap $\Gamma$, choose
$H$ by doubling candidate integers until
\begin{equation}\label{eq:fixed-phase-horizon-choice}
  B(H)<\frac{3H\Gamma}{8\G}.
\end{equation}
Then freeze $H$ and all the selector parameters throughout that phase.
Such a choice exists and satisfies
$H=\bigOtilde_{\p,\q}(1+(\G\RR/\Gamma)^{1/\rho})$.
The doubling tests only evaluate the displayed inequality and make no
first-order queries. If $H$ consecutive null transitions were made,
center errors of at most $\Gamma/(16\G)$ would force reference-center
movement of at least $3H\Gamma/(8\G)$, contradicting
\cref{eq:fixed-phase-horizon-budget,eq:fixed-phase-horizon-choice}.
Thus the phase ends within $H+1$ center queries. Choose a new horizon at
the next phase; summing these phase bounds gives
\cref{eq:oracle-compl-from-movement}.
\par

\paragraph{Sampling, stopping, and total cost.}
Use the resulting deterministic query budget
$\bigOtilde_{\p,\q}(1+\eps^{-1/\rho})$, with $\G=\RR=1$ under the
normalization above. The horizons are polynomial in inverse accuracy
for fixed $\p,\q$.
\par%
Terminate and return the current best feasible point if that budget is
exhausted.  Use $\errE=\Esol=\eps/64$ at every center.
On simultaneous sampling success these errors satisfy
\cref{eq:bundle-center-tolerance} throughout, since $\Gamma>\eps$.
The concentration argument in \cref{prop:monte-carlo}, with failure
probability allocated across all center queries, uses polynomially many
samples for fixed $\p,\q$.  Denote their deterministic total budget by $M$.

For the untruncated independent Gaussian draws,
\[
  \Prb\!\left\{
    \max_{1\le j\le M}\norm{\Z_j}_{\infty}
      >\sqrt{2\log(4\d M/\delta)}
  \right\}\le\delta/2.
\]
If a draw exceeds this threshold, return the current best feasible point
before solving its sample problem.  Otherwise all sampled objectives have
polynomial Lipschitz bounds.  Allocate the remaining $\delta/2$ to center
estimation failures.  Fresh sampling makes each estimate valid
conditionally on the preceding history, so a union bound gives joint
success with probability at least $1-\delta$.
On that event, \cref{lem:buffered-implementation} guarantees certified
error at most $\eps$ within the chosen query budget.

All computed query points and both early-return rules preserve feasibility.
The planned horizon, sample count, and inverse objective tolerances are
polynomial in $\d$, $\eps^{-1}$, and $\log(1/\delta)$ for fixed norm
parameters.  The preceding convex solves therefore give the claimed
total arithmetic bound on every run, while the accuracy guarantee holds
on the joint success event.

\end{proof}

\section{Proofs of final results from \Cref{sec:setting-main-results}}

\begin{proof}\linkofproof{thm:movement-main}
For $\p<\min\{\q,2\}$, the exact mixed-norm Gaussian center in
\cref{eq:gaussian-center} is feasible and satisfies \cref{eq:movement-main}
deterministically by \cref{prop:low-p-center}.
For $\p\ge\min\{\q,2\}$, use the potential-energy center:
\cref{prop:energy-complement} gives the stated bound, including the
diagonal $\p=\q$ and both endpoints.
For the computational claim, it remains to consider $\p<\q$.
For the Monte Carlo approximation when $\p<2$, apply
\cref{prop:monte-carlo} with
$\errE=\Esol=(\RR/4)\T^{-1/\p+(1/\q-1/2)_{\pospart}}$.
It preserves feasibility and gives \cref{eq:movement-main} with probability
at least $1-\delta$.  For $\p\ge2$, the energy center already gives the
guarantee with probability one, without sampling.
It remains to justify the computational claim.  For $\p<2$, it is exactly
\cref{thm:polynomial-implementation}.  For $2\le\p<\q$, the energy center
minimizes the convex function $\norm{x}_{\p}^{\p}/\p$ over a body represented
by the ball and the stored affine cuts.  The separation-oracle argument in
the proof of \cref{thm:polynomial-implementation} computes a feasible point
within $O(\!\RR\T^{-1/\p})$ in $\ell_{\q}$ of that center in polynomial time.
The resulting path differs in total movement by at most
$O(\RR\T^{1-1/\p})$, so \cref{eq:movement-main} is preserved.

\end{proof}

\begin{proof}\linkofproof{thm:optimization-main}
Normalize $\G=\RR=1$ and set $\rho=1/\p-(1/\q-1/2)_{\pospart}$.
{\Cref{prop:low-p-center} gives the reference-center
movement bound.}
\Cref{lem:buffered-implementation} yields
$\bigOtilde_{\p,\q}(1+\eps^{-1/\rho})$ oracle calls whenever the center
errors satisfy \cref{eq:bundle-center-tolerance}.
\Cref{thm:polynomial-implementation} supplies feasible approximations
meeting these tolerances jointly with probability at least $1-\delta$,
with polynomial additional cost.
Choose $\eps=\bigOtilde_{\p,\q}(\T^{-\rho})$ so that the planned oracle
budget is at most $\T$ (if the target exceeds one, one feasible query suffices
up to constants), and return the best queried point if that budget
is exhausted.  Every output is feasible; on the success event its error
is at most $\eps$.  Rescaling gives \cref{eq:optimization-main}.
\end{proof}

\begingroup
\renewcommand{\eta}{\newlink{def:hard-instance-scale}{\oldeta}}
\renewcommand{\Delta}{\newlink{def:hard-offsets}{\oldDelta}}
\renewcommand{\nu}{\newlink{def:hard-distribution}{\oldnu}}
\renewcommand{\delta}{\newlink{def:hard-offsets}{\olddelta}}

\section{Proofs from \Cref{sec:lower-bounds}}

\label{app:LB-randomized}

The proof of the randomized lower bound reported in \Cref{thm:LB_randomized_pq} works in much broader generality, which we present here for completeness.

\begin{theorem}\label{thm:LB-randomized-fat}%
Let $\cX,\cG\subseteq E$ be convex with $\cG=-\cG$, $\eps>0$, $\beta\in[0,1)$, and $K\leq \fat_{2\eps}(\cL(\cX,\cG))$. There exists a probability distribution $\newtarget{def:hard-distribution}{\nu}$ over $\cF(\cX,\cG)$ for which any deterministic algorithm $A$ making $\T$ queries to a local oracle (here $\T$ is a random variable) and providing $\widehat{x}\in \cX$ with
\[ \Prb_{f\sim \nu}\Big[ f(\widehat{x})-\min_{x\in \cX}f(x) >\eps \Big] \leq \beta\]
must satisfy
\[ \E_{f\sim\nu}[\T] \geq \frac{(1-\beta)K-1}{2}.\]
In particular, for nontrivial failure probability, $0\leq \beta\leq 1/2$, we have that $ \E_{f\sim\nu}[\T] = \Omega(K)$.
\end{theorem}

\begin{proof}
    Let $g_1,\ldots,g_K\in\cG$ be a set that is $(2\eps)$-shattered by $\cL(\cX,\cG)$ with thresholds $r_1,\ldots,r_K$. For $s\in\{-1,+1\}^K$, let $x_s\in \cX$ be the corresponding witness, from \cref{def:fat-shattering}. Hence
    \begin{equation}\label{eq:shatter-gamma}
        \gamma\defi \min_{s\in\{-1,+1\}^K} \min_{i\in[K]} s_i(\langle g_i,x_s\rangle-r_i) >\eps.
    \end{equation}
    Let $\newtarget{def:hard-instance-scale}{\eta}= \frac{\gamma-\eps}{4}$ and $c=\eta-\gamma$. Our instances will be uniformly drawn from the following family of objectives:
    \begin{equation} \label{eq:string-guessing-instances}
    f_{s,\Delta}(w) = \max\Big\{c,\max_{i\in[K]} -s_i(\langle g_i,w\rangle-r_i-\Delta_i)\Big\}, \quad s\in\{-1,+1\}^K,\ \newtarget{def:hard-offsets}{\Delta}\in [-\eta,\eta]^K.
    \end{equation}
    In particular $\nu$ is the law of $f_{S,\Delta}$ when $S\sim \mbox{Unif}(\{-1,+1\}^K)$, $\Delta\sim\mbox{Unif}([-\eta,\eta]^K)$ independently.

    \paragraph{Packing property.} We first show that solving the optimization problem on an instance determines the string parameter $s$ (a.k.a.~the packing property \citep{Braun:2017}). We note that by \cref{eq:shatter-gamma} the shattering witness $x_s$ satisfies
    \[ -s_i(\langle g_i,x_s\rangle-r_i-\Delta_i) \leq \eta-\gamma=c. \]
    In particular, $\min_{x\in \cX} f_{s,\Delta}(x)=c$.
    Let now $\widehat{x}$ be a $\eps$-optimal solution to the minimization of $f_{s,\Delta}$. Then $f_{s,\Delta}(\widehat{x})\leq c+\eps=-3\eta$. In particular
    \[ 0 < 2\eta \leq  3\eta+s_i\Delta_i \leq -f_{s,\Delta}(\widehat x)+s_i\Delta_i \leq s_i(\langle g_i, \widehat x\rangle-r_i). \]
    Here the second inequality uses $\eps$-optimality, and the last one uses the definition of $f_{s,\Delta}$. We conclude that
    \[ s_i = \sgn(\langle g_i,\widehat x\rangle-r_i) \quad \forall i\in[K],\]
    i.e., $\widehat x$ determines $s$ uniquely.

    \paragraph{String-guessing oracle emulation.} We perform now a {\em string guessing-oracle} emulation, following \citep{Braun:2017}. This corresponds to emulating the answer given by a {\em single-coordinate oracle} (that is, an oracle for piecewise-affine functions that provides answers whose subgradient coincides with that of an affine piece), by an oracle that only provides the values of some coordinates of $s$. %
    Fix the realization $\Delta=\delta$ and make it available to the string-guessing simulator. In the local-oracle simulation below, this information is kept private from the optimization algorithm.
    Let $x\in \cX$ be an arbitrary query, and consider the scores
    \[ b_i = \langle g_i,x\rangle - r_i-\delta_i \quad (\forall i\in[K]),\]
    and select a deterministic ordering $\sigma\in \mathfrak S_K$ (the permutation group) such that $|b_{\sigma(1)}|\geq \ldots\geq |b_{\sigma(K)}|$. Query the string-guessing oracle with $t_j=\sgn(b_{\sigma(j)})$, using a fixed convention when  $b_{\sigma(j)}=0$. Let $k$ be the first index in $[K]$ such that $t_k\neq s_{\sigma(k)}$, then $-s_{\sigma(k)}b_{\sigma(k)}=|b_{\sigma(k)}|$, while every later coordinate $j\geq k$ is such that $-s_{\sigma(j)}b_{\sigma(j)}\leq -s_{\sigma(k)}b_{\sigma(k)}$. Hence, the oracle returning $(-s_{\sigma(k)}b_{\sigma(k)},-s_{\sigma(k)}g_{\sigma(k)})$ is a valid single-coordinate first-order oracle answer in this case. If the answer by the string-guessing oracle is \textsc{equal}, then the entire string is known and one can compute a single-coordinate first-order oracle answer directly. Hence, the string-guessing oracle provides an emulation of a single-coordinate first-order oracle.

      \paragraph{Lower bound for arbitrary local oracles.}
    By the packing property, every $\eps$-optimal output
    by the algorithm identifies the random string $S$.
    To extend the emulation to an arbitrary local oracle, consider
    the maximal local oracle from
    \citep[Definition~VI.1]{Braun:2017}, from which every local
    oracle can be emulated \citep[Lemma~VI.2]{Braun:2017}.

    Define
    \[
        P_i(x)=r_i-\langle g_i,x\rangle-\eta,\qquad
        U_i=\Delta_i+\eta,\qquad
        H_{s,u}(z)=\max_{i\in[K]}s_i(z_i+u_i).
    \]
    The coordinates of $U$ are independent and uniform on
    $[0,2\eta]$, independently of $S$, and
    \[
        f_{S,\Delta}(x)=\max\{c,H_{S,U}(P(x))\}.
    \]
    The behavior of $H_{S,U}$ in a neighborhood of $P(x)$
    determines the behavior of $f_{S,\Delta}$ in a neighborhood of $x$, by composition with the known affine map $P$ and the function $v\mapsto\max\{c,v\}$. This also holds when $H_{S,U}(P(x))=c$. 
    Thus an algorithm $A$ using a local oracle for $f_{S,\Delta}$
    can be simulated, with the same number of queries, by an
    algorithm $B$ using the maximal local oracle for $H_{S,U}$
    at points in $P(\cX)$. Neither algorithm is supplied the
    offset vector.

    The argument of
    \citep[Lemma~VI.4 and the subsequent proof of Theorem~VI.3,
    arXiv v3]{Braun:2017} applies to $H$: its conditional-density
    induction uses only the coordinate perturbations and the
    adaptive query history, and applies unchanged to queries in
    $P(\cX)$ and perturbations uniform on $[0,2\eta]$.
    Almost surely, each query has either only previously exposed
    maximizing pieces or a unique maximizing piece that has not
    been exposed before. As in the subsequent simulation in that
    proof, the maximal-oracle reply is therefore reconstructed
    from one maximizing-piece reply together with the previously
    exposed pieces; when old pieces tie, their maximum determines
    the germ.

    The preceding string query supplies a maximizing affine piece
    of $H$, or determines the whole sign string.
    Consequently, the string-guessing simulator reproduces the
    replies received by $B$, and hence by $A$, using at most one
    string query per optimization query. In this simulation the
    offsets may be known privately to the simulator, but are not
    supplied to $A$ or $B$.

    Since $S$ has finite support, the preceding probability-one
    event holds for every sign string for almost every fixed
    $\delta$. For such a $\delta$, let 
    \(
        \beta_\delta
        =\Prb_S\bigl[f_{S,\delta}(\widehat x)-c>\eps\bigr].
    \) 
    By the packing property, the resulting string-guessing
    algorithm has error at most $\beta_\delta$.
    Proposition~III.3 of \citep{Braun:2017}, applied to this
    simulation (cf. Lemma~III.5 there), gives
    \[
        \E[\T\mid\Delta=\delta]
        \geq\frac{(1-\beta_\delta)K-1}{2},
    \]
    with the case $\beta_\delta=1$ following trivially from
    nonnegativity. Since $\E_\Delta[\beta_\Delta]\leq\beta$,
    averaging over $\Delta$ yields
    \[
        \E[\T]
        \geq\frac{(1-\E_\Delta[\beta_\Delta])K-1}{2}
        \geq\frac{(1-\beta)K-1}{2}.
    \]
    This proves the theorem.
 
\end{proof}

\begin{proof}\linkofproof{thm:LB-randomized-budget}
Fix a local oracle $\oracle$ and a query budget $\T$.

\paragraph{Converting query-complexity bounds to risk bounds.}
Suppose that, for some $\eps>0$, a distribution $\nu$ over
the objective class has the following property: every
deterministic algorithm whose output is $\eps$-optimal with
probability at least $3/4$ under $\nu$ must make more than
$\T$ queries in expectation. Then
\begin{equation}\label{eq:complexity-to-risk}
    \operatorname{Risk}_{\T}
    \bigl(\oracle,\cF(\cX,\cG)\bigr)
    \geq \frac{\eps}{4}.
\end{equation}
Indeed, let $A\in\RandAlg_{\T}(\oracle)$, and let 
\(
    Z(f,A)
    :=f\bigl(\widehat x_{\T}(A,f)\bigr)
      -\min_{x\in\cX}f(x).
\) 
If $\mathbb E_{f\sim\nu,A}[Z(f,A)]\leq\eps/4$, Markov's
inequality gives
\[
    \Prb_{f\sim\nu,A}[Z(f,A)>\eps]\leq\frac14.
\]
Averaging over the internal random seed of $A$, there is
a fixed seed for which the resulting deterministic algorithm
has $\nu$-failure probability at most $1/4$. It still makes
at most $\T$ queries, contradicting the property of $\nu$.
Therefore
\[
    \sup_{f\in\cF(\cX,\cG)}\mathbb E_A[Z(f,A)]
    \geq
    \mathbb E_{f\sim\nu,A}[Z(f,A)]
    >\frac{\eps}{4}.
\]
Taking the infimum over $A$ proves
\eqref{eq:complexity-to-risk}.

\paragraph{The case $\p<\infty$.}
We use \cref{eq:fat_shattering_values,eq:fat-shattering-11}
in their inverse, cardinality--margin form. Namely, there
are constants $a_{\p,\q},b_{\p,\q}>0$ such that, for every
integer $1\leq m\leq b_{\p,\q}\d$,
\[
    \fat_{\alpha_m}\bigl(
        \cL(\B_{\p}^{\d},\B_{\q^{\conj}}^{\d})
    \bigr)\geq m,
    \qquad
    \alpha_m:=a_{\p,\q}
    \begin{cases}
        m^{-1/\p+(1/\q-1/2)_{\pospart}},
            & \p<\q,\\[2mm]
        \d^{1/\q-1/\p}m^{-1/\max\{\q,2\}},
            & \q\leq\p<\infty.
    \end{cases}
\]

Set $K=8\T$ and 
\(
    \eps:=\frac{\G\RR}{2}\alpha_K.
\) 
Choosing $c_{\p,\q}\leq b_{\p,\q}/8$ ensures that
$K\leq b_{\p,\q}\d$. Scaling then gives
\[
    \fat_{2\eps}\bigl(
        \cL(\RR\B_{\p}^{\d},\G\B_{\q^{\conj}}^{\d})
    \bigr)
    =
    \fat_{\alpha_K}\bigl(
        \cL(\B_{\p}^{\d},\B_{\q^{\conj}}^{\d})
    \bigr)
    \geq K.
\]
By \cref{thm:LB-randomized-fat}, with $\beta=1/4$, there
is a distribution $\nu$ under which every deterministic
algorithm with failure probability at most $1/4$ requires
at least
\[
    \frac{(3/4)K-1}{2}
    =3\T-\frac12>\T
\]
expected queries. Applying \eqref{eq:complexity-to-risk}
yields
\[
    \operatorname{Risk}_{\T}\bigl(
        \oracle,
        \cF(\RR\B_{\p}^{\d},\G\B_{\q^{\conj}}^{\d})
    \bigr)
    \geq\frac{\G\RR}{8}\alpha_{8\T}.
\]
Substituting the expression for $\alpha_{8\T}$ and
decreasing $c_{\p,\q}$ if necessary proves the two
finite-$\p$ cases.

\paragraph{The case $\p=\infty$.}
For $\p=\q=\infty$, the box lower bound
\citep[Theorem~VI.5]{Braun:2017} provides, for sufficiently
small $\eps/(\G\RR)$, a distributional query lower bound
\[
    \Omega\!\left(
        \d\log\frac{\G\RR}{\eps}
    \right)
\]
for arbitrary local oracles, including algorithms with
failure probability at most $1/4$.
Choose $\eps=a_0\G\RR$ for a sufficiently small universal
constant $a_0>0$. This lower bound exceeds $\T$ when
$\T\leq c_{\p,\q}\d$, after decreasing $c_{\p,\q}$ if
necessary. By \eqref{eq:complexity-to-risk},
\[
    \operatorname{Risk}_{\T}\bigl(
        \oracle,
        \cF(\RR\B_{\infty}^{\d},\G\B_1^{\d})
    \bigr)
    \geq\frac{a_0\G\RR}{4}.
\]
Finally, $\B_1^{\d}\subseteq\B_{\q^{\conj}}^{\d}$ for every
$\q$, so the same hard instances are admissible for
$\cF(\RR\B_{\infty}^{\d},\G\B_{\q^{\conj}}^{\d})$.
This proves the last case and completes the proof.
\end{proof}

\endgroup

\printbibliography

\end{document}